\documentclass[12pt]{amsart}

\usepackage{amsmath}
\usepackage{amssymb}
\usepackage{amsthm}
\usepackage{mathrsfs}
\usepackage{hyperref}
\usepackage{esint}
\usepackage{enumerate}

\usepackage{comment}

\DeclareMathOperator{\Div}{div}

\DeclareMathOperator{\tr}{tr}
\DeclareMathOperator{\supp}{spt}
\DeclareMathOperator{\sym}{Sym}
\DeclareMathOperator{\Cyl}{Cyl}
\DeclareMathOperator{\diam}{diam}

\renewcommand{\AA}{\textit{\r{A}}}
\newcommand{\R}{\ensuremath{\mathbb{R}}}

\newtheorem{theorem}{Theorem}[section]
\newtheorem{lemma}[theorem]{Lemma}
\newtheorem{proposition}[theorem]{Proposition}
\newtheorem{corollary}[theorem]{Corollary}
\newtheorem{remark}[theorem]{Remark}

\newdimen\slantmathcorr
\def\oversl#1{
\setbox0=\hbox{$#1$}
\slantmathcorr=\wd0
\hskip 0.2\slantmathcorr \overline{\hbox to 0.8\wd0{%
\vphantom{\hbox{$#1$}}}}
\hskip-\wd0\hbox{$#1$}
}

\def\undersl#1{
\setbox0=\hbox{$#1$}
\slantmathcorr=\wd0
\underline{\hbox to 0.8\wd0{%
\vphantom{\hbox{$#1$}}}}
\hskip-0.8\wd0\hbox{$#1$}
}

\title[Sharp one-sided curvature estimates]{Sharp one-sided curvature estimates for fully nonlinear curvature flows and applications to ancient solutions}
\author{Mat Langford and Stephen Lynch}

\begin{document}

\begin{abstract}
We prove several sharp one-sided pinching estimates for immersed and embedded hypersurfaces evolving by various fully nonlinear, one-homogeneous curvature flows by the method of Stampacchia iteration. These include sharp estimates for the largest principal curvature and the inscribed curvature (`cylindrical estimates') for flows by concave speeds and a sharp estimate for the exscribed curvature for flows by convex speeds. 
Making use of a recent idea of Huisken and Sinestrari, we then obtain corresponding estimates for ancient solutions. In particular, this leads to various characterisations of the shrinking sphere amongst ancient solutions of these flows.
\end{abstract}

\maketitle
\tableofcontents

\section{Introduction}

Let $M$ be a smooth, closed manifold of dimension $n \geq 2$ and $I \subset \mathbb{R}$ an interval. We are interested in smooth families of smooth, orientable immersions $X:M\times I \to \mathbb{R}^{n+1}$ that evolve according to the equation
\begin{equation}
\label{eq:CF}
\tag{CF}
\partial_t X(x,t) = -F(x,t)\nu (x,t)\,.
\end{equation}
Here, $\nu(\cdot,t)$ is a smoothly time-dependent unit normal vectorfield for the immersion $X_t:=X(\cdot,t)$ and $F(x,t)$ depends purely on the principal curvatures $\kappa_1(x,t)\leq \dots\leq \kappa_n(x,t)$ (the eigenvalues of the shape operator $A=D\nu$ at $(x,t)$) of $X_t$ at $x$ in the following way: We assume there is an open, symmetric cone $\Gamma \subset \mathbb{R}^n$ and a smooth, symmetric, one-homogeneous function $f:\Gamma \to \mathbb{R}_+$ such that $F(x,t) = f(\kappa(x,t))$, where $\kappa$ denotes the $n$-tuple $(\kappa_1,\dots,\kappa_n)$. Symmetry ensures that $F$ may equally be viewed as a smooth, orthonormal basis-invariant function of the components of $A$ \cite{Gl}. We further require that $f$ be monotone increasing in each argument---an ellipticity condition for \eqref{eq:CF} which, in particular, guarantees the short-time existence of a solution of \eqref{eq:CF} starting from any smooth, closed initial immersion on which $\kappa$ takes values in $\Gamma$ \cite{HuPol}. We will refer to such functions $f:\Gamma\to\R$ as \emph{admissible speeds}. Unless $n=2$, we usually require that $\Gamma$ is a convex cone and $f$ is either concave or convex. Flows in this class tend to exhibit similar behaviour to that of the mean curvature flow.

\subsection{Uniform ellipticity} Just as is the case for solutions of the mean curvature flow, comparison with shrinking sphere solutions indicates that compact solutions of \eqref{eq:CF} necessarily become singular in finite time \cite[Theorem 4.32 and Remarks 4.6]{L}. On the other hand, due to the fully nonlinear nature of \eqref{eq:CF}, characterising the singular time by curvature blow-up is a subtle issue; indeed, a poor choice of the domain $\Gamma$ can lead to unnatural `type-0' singularities, whereby the curvature $n$-tuple $\kappa$ of the solution leaves $\Gamma$. Since we are interested here in proving a priori estimates, we shall always simply assume that $\kappa$ stays inside the closure of some open, symmetric cone $\Gamma_0$ satisfying $\Gamma_0\Subset \Gamma$, by which we mean $\overline\Gamma_0\cap S^n\subset \Gamma$. Under this `uniform ellipticity' condition, the problem of long-time regularity is reduced to the deep problem of obtaining second-derivative H\"{o}lder estimates for solutions of fully non-linear, uniformly parabolic PDE. Such results follow from the celebrated Harnack inequality of Krylov and Safonov in case $f$ is either concave or convex \cite{Kry81}; a more recent result of Andrews \cite{An04} covers the case $n=2$ without the need for additional concavity conditions.




We note that, in particular, it is always possible to find preserved cones $\Gamma_0\Subset\Gamma$ in the following situations:
\begin{itemize}
\item[--] $n=2$ \cite{An10,ALM15}
\item[--] $f:\Gamma\to\R$ is convex and $\Gamma_+\subset \Gamma$ \cite{L}
\item[--] $f:\Gamma\to\R$ is concave and either $\Gamma=\Gamma_+$ and $f$ is inverse-concave\footnote{Meaning that the function $f_\ast:\Gamma_+\to\R$ defined by $f_\ast(z_1^{-1},\dots,z_n^{-1}):=f(z)^{-1}$ is concave.} \cite{An07}, or $f|_{\partial\Gamma}=0$ \cite{An94a},
\end{itemize}
where $\Gamma_+:=\{z\in \R^n:z_i>0\,\, \forall\, i=1,\dots,n\}$ is the positive cone. See \cite{AMY} for a more comprehensive discussion of this important issue in case $\Gamma=\Gamma_+$.



It is worth noting that the class of admissible flows is quite large. For example, the class of convex admissible flow speeds includes the mean curvature, the power means with power $r\geq 1$ and 1-homogeneous convex combinations of these \cite{ALM14}. The class of concave admissible flow speeds includes the mean curvature, the power means with power $r\leq 1$, one-homogeneous roots of ratios of the elementary symmetric polynomials and 1-homogeneous concave combinations of these \cite{An07}. For a characterisation of admissible flow speeds when $n=2$, see \cite{ALM15}.

\subsection{One-sided curvature pinching}

When $F$ is given by the mean curvature, $H$, much more is known about the geometry of solutions near a singularity. In \cite{Huisk84}, Huisken demonstrated that any solution of mean curvature flow which is strictly convex contracts to a round point. A key ingredient in the proof was the umbilic estimate, which states that, for any $\varepsilon >0$, there exists a constant $C_\varepsilon>0 $ depending only on $\varepsilon$ and the initial (convex) hypersurface such that
\[|\AA|^2\leq \varepsilon H^2 + C_\varepsilon\,,\]
where $|\AA|^2=|A|^2-\frac{1}{n}H^2$ is the squared norm of the trace-free second fundamental form, $\AA:=A-\frac{1}{n}Hg$. The umbilic estimate shows that any point at which $H$ blows up becomes umbilic in the limit. To prove this estimate, Huisken used Simons' identity to derive $L^p$ estimates for the function $G_\sigma:=|\AA|^2H^{\sigma-2}$ for $p\gg 1$ and $\sigma\sim p^{-\frac{1}{2}}$. Stampacchia iteration, with the help of the Sobolev inequality of Michael and Simon, was then used to translate these into an $L^\infty$ estimate. Huisken and Sinestrari \cite{HuSi99a,HuSi99b} later used similar techniques to prove the convexity estimate (see also \cite{Wh03}),
\[\kappa_1 \geq - \varepsilon H - C_\varepsilon\,,\]
for mean-convex ($H > 0$) solutions of mean curvature flow, which shows that mean convex solutions become weakly convex at a singularity. Interpolating between the umbilic and convexity estimates are the cylindrical estimates \cite{HuSi99a},
\begin{equation}
\label{cylindrical_H}
|A|^2 - \frac{1}{n-m}H^2 \leq \varepsilon H^2 + C_\varepsilon,
\end{equation}
which hold on $(m+1)$-convex solutions, that is, those satisfying $H>0$ and $\kappa_1 + \dots + \kappa_{m+1} \geq \alpha H$ for some $\alpha >0$. The left-hand side of \eqref{cylindrical_H} is non-positive only at points where $\kappa_1 + \dots + \kappa_m > 0$ or $0 = \kappa_1=\dots=\kappa_m$ and $0<\kappa_{m+1}=\dots=\kappa_n$; so, again, these estimates place marked restrictions on the geometry of solutions at a singularity. 

These results are now known to hold in much greater generality: By carefully constructing pinching functions adapted to the speed function, and adapting the arguments of Huisken and Huisken--Sinestrari, it is possible to obtain analogous convexity estimates for solutions of \eqref{eq:CF} whenever $F$ is a convex admissible speed \cite{ALM14}. Moreover, the convexity assumption for the speed can be removed when $n=2$ \cite{ALM15}. 
In addition, for flows by convex admissible speeds, a family of cylindrical estimates holds for solutions that are $(m+1)$-convex, in the sense that $\kappa_1+\dots+\kappa_{m+1} \geq \alpha F$ for some $\alpha>0$ \cite{AnLa14}. Umbilic estimates have been obtained for flows by convex admissible speeds, certain concave admissible speeds and flows of surfaces by any admissible speed \cite{Cw85,Cw87,L}. 
More recently, a cylindrical estimate analogous to the $m=1$ case of \eqref{cylindrical_H} has recently been proven for flows of two-convex hypersurfaces by a class of concave speeds which includes the two-harmonic mean curvature \cite{HuBre}. In this special situation, the cylindrical estimate actually implies a convexity estimate. In the present work, we combine techniques from \cite{ALM14} and \cite{HuBre} to prove a family of cylindrical estimates for more general concave speeds.

Given $m\in\{0,\dots,n-1\}$, denote by $\Gamma_{m+1}$ the open, convex cone
\[
\Gamma_{m+1} := \bigcap_{\sigma \in P_n} \{z \in \mathbb{R}^n : z_{\sigma(1)} + \dots + z_{\sigma(m+1)} >0\}\,,
\]
where $P_n$ is the set of permutations of $\{1,\dots,n\}$. Recall that, given two open cones $\Gamma_0$, $\Gamma\subset \R^n$, we write $\Gamma_0\Subset \Gamma$ if $\overline \Gamma_0\cap S^{n}\subset \Gamma$.

\begin{theorem}
\label{cylindrical_ests}
Let $f:\Gamma\to\R_+$ be a concave admissible speed. Let $X:M\times [0,T)\to \mathbb{R}^{n+1}$ be a compact solution of \eqref{eq:CF} on which $\kappa$ takes values in some open, symmetric cone $\Gamma_0 \Subset \Gamma\cap\Gamma_{m+1}$, $m\in \{0,\dots,n-1\}$. Then for every $\varepsilon >0$ there exists $C_\varepsilon = C_\varepsilon(n,f,\Gamma_0,X_0,\varepsilon)$ such that
\[\kappa_n - c_m F \leq \varepsilon F + C_\varepsilon\quad \text{on} \quad M \times [0,T)\,,\]
where $c_m^{-1}$ is the value $F$ takes on the cylinder $\R^m\times S^{n-m}$:
\[c_m^{-1} := f(\underbrace{0, \dots, 0}_{m\text{\emph{-times}}}, 1, \dots, 1).\]
\end{theorem}
\begin{remark}\mbox{}
\begin{itemize}
\item If $m=n-1$, then the inclusion $\Gamma\subset\Gamma_{m+1}$ already follows from concavity of $f$.
\item The constant $C_\varepsilon$ can be written as $\tilde C_\varepsilon R^{-1}$, where \[\tilde C_\varepsilon = \tilde C_\varepsilon (n,f,\Gamma_0,V,\tau,\varepsilon)\] is scaling invariant, the scale parameter $R^{-1}$ is an upper bound for $\max_{M\times\{0\}}F$, $VR^n$ is an upper bound for the initial area $\mu_0(M)$ and $\tau R^2$ is an upper bound for the maximal time $T$. (Up to a constant depending on $n$ and $\Gamma_0$) we can take $\tau R^2$ to be a bound for the square of the initial circum-radius or the inverse square of the initial minimum of $F$.
\end{itemize}
\end{remark}

Observe that, even in the special two-convex case, our estimate is slightly sharper than the one proved in \cite{HuBre}. This will be important in \S \ref{sec_insc}, where we use Theorem \ref{cylindrical_ests} to prove a sharp estimate for the inscribed curvature using Stampacchia iteration.

\subsection{Inscribed and exscribed curvature pinching}

Suppose that $M$ is properly embedded in $\mathbb{R}^{n+1}$ and bounds a precompact open set $\Omega\subset \R^{n+1}$, and let $\nu$ denote the outward pointing unit normal on $M$. The \emph{inscribed curvature} $\overline k(x)$ at $x\in M$ is then defined as the curvature of the boundary of the largest ball contained in $\Omega$ which has first order contact with $M$ at $x$. By a straightforward computation \cite{An12,ALM13}, $\overline k$ is given by
\begin{equation}
\label{insc_def}
\overline k (x) := \sup_{y\in M\setminus\{x\}} k(x,y),
\end{equation}
where $k:(M\times M)\setminus\{(x,x):x\in M\}\to\R$ is given by 
\[k(x,y) := \frac{2\langle x-y, \nu(x)\rangle}{|x-y|^2}\,.\] 
If the supremum in \eqref{insc_def} is attained at some $y\neq x$, then there is a ball contained in $\Omega$ which is tangent to $M$ at $x$ and $y$ and $\overline k(x)$ equals the curvature of its boundary sphere. Otherwise,
\[\overline k(x) = \limsup_{y \to x} k(x,y) = \sup_{y \in T_xM} \frac{A_x(y,y)}{g_x(y,y)}\,.\]
Similarly, the \emph{exscribed curvature} $\underline k$ is defined as the curvature of the boundary of the largest ball, halfspace or ball-complement (equipped with their outward pointing normals) which is contained in the complement of $\overline\Omega$ and has first order contact with $M$ at $x$. Equivalently, $\underline k$ is given by
\[\underline k (x) := \inf_{y \in \setminus\{x\}} k(x,y)\,.\]
We note that changing the orientation of $M$ interchanges $\overline k $ and $\underline k$, and that each of the equalities $\overline k = \frac{1}{n}H$ and $\underline k = \frac{1}{n} H$ characterises the round spheres.

In a major recent breakthrough, Andrews \cite{An12} showed, using only the maximum principle, that the maximum of $\overline k/H$ is non-increasing along an embedded, mean convex solution of mean curvature flow, while the minimum of $\underline k/H$ is non-decreasing, providing a simple proof of earlier `non-collapsing' results of White \cite{Wh03} and Sheng--Wang \cite{ShWa09}. 

These estimates have been referred to as \emph{interior} and \emph{exterior non-collapsing} estimates respectively. For flows of convex hypersurfaces, a straightforward blow-up argument and the strong maximum principle show that the collapsing ratios $\overline k/H$ and $\underline k/H$ actually improve, which yields a unified proof of the convergence results of Huisken \cite{Huisk84} and Gage--Hamilton \cite{GaHa86}. Using the convexity estimate and Stampacchia iteration, Brendle was able to prove that the collapsing ratios also improve to a sharp value at a singularity under mean convex mean curvature flow; namely,
\begin{equation}\label{inscribed_improves_MCF}
\overline k-H \leq \varepsilon H + C_\varepsilon
\end{equation}
and 
\begin{equation}\label{exscribed_improves_MCF}
\underline k \geq - \varepsilon H - C_\varepsilon\,.
\end{equation}
Recently, a similar result for the inscribed curvature was proved by Brendle and Hung for the two-harmonic mean curvature flow \cite{Hung}. 

In a remarkable recent development, Haslhofer and Kleiner have obtained local curvature estimates for embedded mean convex mean curvature flow \cite{HK1}. In particular, these curvature estimates can be used to prove \eqref{inscribed_improves_MCF} and \eqref{exscribed_improves_MCF} by a blow-up argument \cite{HK2}. Subsequently, the inscribed curvature estimate estimate \eqref{inscribed_improves_MCF} has been extended to the family of estimates
\begin{equation}\label{insc_improves_MCF_cylindrical}
\overline k-\frac{1}{n-m}H \leq \varepsilon H + C_\varepsilon
\end{equation}
for $(m+1)$-convex solutions, $m \leq n-1$, first using a blow-up argument (making use of the Haslhofer--Kleiner estimates) \cite{La15}, and then using Stampacchia iteration \cite{Lan16}.

For fully nonlinear flows, the non-collapsing picture is not so straightforward: whereas flows by concave speeds are interior non-collapsing and flows by convex speeds are exterior non-collapsing \cite{ALM13}, two-sided non-collapsing is only known to hold in certain special cases: For the mean curvature flow (the mean curvature being a linear function of the curvatures) and flows of convex hypersurfaces by concave, \emph{inverse-concave} speeds \cite{NC2}. This provides a large class of two-sided non-collapsing flows of \emph{convex} hypersurfaces but none, beyond the mean curvature flow, of non-convex hypersurfaces. On the other hand, we cannot hope for too much here: Since even flows by concave speeds may not preserve convexity of solutions \cite{AMY}, we cannot expect their exscribed curvature to behave particularly well. In particular, this means that the Haslhofer--Kleiner theory is not available to us (indeed, their theory also makes use of additional results which are known only for the mean curvature flow, such as Huisken's monotonicity formula and White's $\varepsilon$-regularity theorem); however, the Stampacchia iteration method is still available. In Section \ref{sec_insc}, we use it to prove a sharp inscribed curvature estimate for flows by concave speeds.

\begin{theorem}
\label{inscribed_improves}
Let $f:\Gamma\to\R_+$ be a concave admissible speed. Let $X:M\times [0,T)\to \mathbb{R}^{n+1}$ be a compact, embedded solution of \eqref{eq:CF} on which $\kappa$ takes values in some open, symmetric cone $\Gamma_0 \Subset \Gamma \cap \Gamma_{m+1}$, where $m\in \{0,\dots,n-1\}$. Then for every $\varepsilon >0$ there is a constant $C_\varepsilon = C_\varepsilon(n,f,\Gamma_0,M_0,\varepsilon)$ such that
\[\overline k - c_m F \leq \varepsilon F + C_\varepsilon\]
on $M \times [0,T)$. 
\end{theorem}
\begin{remark}
The constant $C_\varepsilon$, as for Theorem \ref{cylindrical_ests}, can be written as $C_\varepsilon=\tilde C_\varepsilon(n,f,\Gamma_0,V,\tau,\Lambda,\varepsilon)R^{-1}$, where $\Lambda$ is the initial interior non-collapsing ratio: $\Lambda:=\max_{M\times\{0\}}\overline k/F$.
\end{remark}

We take a similar approach to Brendle \cite{Brend15} but make use of Theorem \ref{cylindrical_ests} in place of the Huisken--Sinestrari convexity estimate (which is not available to us). This is similar to the approach taken in \cite{Lan16}.

In \S\ref{sec_exsc}, we prove a sharp estimate for the exscribed curvature under flows by convex speeds. 

\begin{theorem}
\label{exscribed_improves}
Let $f:\Gamma\to\R_+$ be a convex admissible speed. Let $X:M\times [0,T)\to \mathbb{R}^{n+1}$ be a compact embedded solution of \eqref{eq:CF} on which $\kappa$ takes values in some open, symmetric cone $\Gamma_0 \Subset \Gamma$. Then for every $\varepsilon >0$ there is a constant $C_\varepsilon = C_\varepsilon(n,f,\Gamma_0,M_0,\varepsilon)$ such that
\[\underline k \geq -\varepsilon F - C_\varepsilon\]
on $M \times [0,T)$. 
\end{theorem}
\begin{remark}
The constant $C_\varepsilon$, as for Theorems \ref{cylindrical_ests} and \ref{inscribed_improves}, can be written as $C_\varepsilon=\tilde C_\varepsilon(n,f,\Gamma_0,V,\tau,\Upsilon,\varepsilon)R^{-1}$ where $\Upsilon$ is the initial exterior non-collapsing ratio $\Upsilon:=\min_{M\times\{0\}}\underline k/F$.
\end{remark}


\subsection{Ancient solutions} In the final section, we consider solutions defined on the time interval $I = (-\infty, T)$, known as ancient solutions. We will only consider compact solutions so that, without loss of generality, we can assume that $T=1$ is the maximal time. Such solutions can be expected to enjoy rigidity results, since diffusion is allowed an infinite amount of time to act. For mean curvature flow, Daskolopuolos--Hamilton--\v Se\v sum \cite{DaHaSe}, Huisken--Sinestrari \cite{HuSi15}, Haslhofer--Hershkovits \cite{HaHe} and (for flows in the sphere) Bryan--Louie \cite{BrLo} and Bryan--Ivaki--Scheuer \cite{BrIvSc} have independently identified various natural geometric conditions under which a convex ancient solution must be a family of shrinking spheres. Some of these results were extended to mean convex ancient solutions in \cite{Lan16}. Motivated by the ideas of Huisken and Sinestrari, we will prove rigidity results for ancient solutions of various fully nonlinear flows. The most important of these is the following:

\begin{theorem}\label{thm:sphere_char}
Let $f:\Gamma_+\to\R$ be an admissible speed. If $f$ is convex or concave, or if $n=2$, then any compact, connected ancient solution $X:M \times (-\infty, 1) \to \mathbb{R}^{n+1}$ of \eqref{eq:CF} which is uniformly convex, in the sense that
\[
\kappa(M\times(-\infty,0])\subset\Gamma_0\Subset \Gamma_+\,,
\]
is a shrinking sphere.
\end{theorem}

If the flow admits a splitting theorem, we are able to deduce, by a blow-down argument, that convex ancient solutions with type-I curvature decay are uniformly convex, and hence shrinking spheres by Theorem \ref{thm:sphere_char} (cf. \cite{HuSi15}).

\begin{theorem}\label{thm:sphere_char_typeI}
Let $f:\Gamma\to\R$, where $\Gamma_+\subset\Gamma$, be an admissible speed satisfying one of the following conditions:
\begin{itemize}
\item $n=2$,
\item $f|_{\Gamma_+}$ is convex, or
\item $f|_{\Gamma_+}$ is concave and inverse-concave
\end{itemize}
and let $X:M\times(-\infty,1)\to\R^{n+1}$ be a compact, connected ancient solution of \eqref{eq:CF} satisfying $\kappa(M\times(-\infty,0])\subset \Gamma_+\cap\Gamma_0$, where $\Gamma_0\Subset\Gamma$. Suppose that $X$ satisfies 
one of the following conditions:
\begin{enumerate}
\item\label{hyp:typeI} Type-I curvature decay:
\[
\limsup_{t\to-\infty}\sqrt{1-t}\max_{M\times\{t\}}F<\infty\,.
\]
\item\label{hyp:bsr} Bounded curvature ratios:
\[
\limsup_{t\to-\infty}\frac{\max_{M\times\{t\}}F}{\min_{M\times\{t\}}F}<\infty\,.
\]
\end{enumerate}
Then $\{M_t\}_{t\in(-\infty,1)}$ is a shrinking sphere.
\end{theorem}
\begin{remark}\mbox{}
\begin{itemize}
\item[--] The technical condition
\[
\kappa(M\times(-\infty,0])\subset \Gamma_+\cap\Gamma_0 \quad \text{for some} \quad \Gamma_0\Subset\Gamma
\]
ensures uniform ellipticity of the flow but not, a priori, uniform convexity. It is superfluous (in that it can be replaced by the weaker condition $\kappa_1>0$) if $\Gamma_+\Subset\Gamma$.
\item[--] In case $f$ is concave, inverse-concavity is actually only required on the faces of $\partial\Gamma_+$, which is is implied by inverse-concavity on $\Gamma_+$ \cite[Remark 1 (6)]{AMY}.
\end{itemize}
\end{remark}

If the flow \eqref{eq:CF} admits a differential Harnack inequality, we are able say even more (cf. \cite{HuSi15}).
\begin{theorem}\label{thm:sphere_char_Harnack}
Let $f:\Gamma\to\R$, where $\Gamma_+\subset\Gamma$, be an admissible speed satisfying one of the following conditions:
\begin{itemize}
\item $n=2$ and $f|_{\Gamma_+}$ is inverse-concave,
\item $f|_{\Gamma_+}$ is convex, or
\item $f|_{\Gamma_+}$ is concave and inverse-concave
\end{itemize}
and let $X:M\times(-\infty,1)\to\R^{n+1}$ be a compact, connected ancient solution of \eqref{eq:CF} satisfying $\kappa(M\times(-\infty,0])\subset \Gamma_+\cap\Gamma_0$, where $\Gamma_0\Subset\Gamma$. Suppose that $X$ satisfies one of the following conditions:
\begin{enumerate}\setcounter{enumi}{2}
\item\label{hyp:ecc} Bounded eccentricity: 
\[
\limsup_{t\to-\infty}\frac{\rho_+(t)}{\rho_-(t)}<\infty\,,
\]
where $\rho_+(t)$ and $\rho_-(t)$ denote, respectively, the circum- and in-radii of $M_t$.
\item\label{hyp:brd} Bounded rescaled diameter:
\[
\limsup_{t\to-\infty}\frac{\diam(M_t)}{\sqrt{1-t}}<\infty\,.
\]
\item\label{hyp:rii} Bounded isoperimetric ratio:
\[
\limsup_{t\to-\infty}\frac{\mu_t(M)^{n+1}}{|\Omega_t|^n}<\infty\,,
\]
where $\Omega_t$ is the region bounded by $M_t$.
\end{enumerate}
Then $\{M_t\}_{t\in(-\infty,1)}$ is a shrinking sphere.
\end{theorem}

For mean curvature flow, it is also possible to obtain a convexity estimate for ancient solutions \cite{Lan16}: any closed, mean-convex ancient solution satisfying a uniform scaling invariant lower curvature bound
\begin{equation*}
\liminf_{t\to-\infty}\min_{M\times\{t\}}\frac{\kappa_1}{H}>-\infty
\end{equation*}
as well as the volume-decay condition
\begin{equation*}
\int_t^0\hspace{-3mm}\int H \leq C (1-t)^\frac{n+1}{2}\quad \text{for all}\quad t<0
\end{equation*}
must be convex. We note that both of these conditions are automatic for ancient solutions satisfying type-I curvature decay. In the present work, we replace the volume decay condition with the related condition
\begin{equation}\label{eq:volume_decay}
\int_t^0\hspace{-3mm}\int F \leq C (1-t)^{r}\quad \text{for all}\quad t<0
\end{equation}
for some $r \ge \frac{n+1}{2}$.

\begin{theorem}\label{thm:anc_cnvx}
Let $f:\Gamma\to\R_+$ be an admissible speed satisfying one of the following conditions:
\begin{itemize}
\item $n=2$ or 
\item $f$ is convex\,.
\end{itemize}
Then any ancient solution $X:M \times (-\infty, 1) \to \mathbb{R}^{n+1}$ of \eqref{eq:CF} which satisfies $\kappa(M\times(-\infty,0])\subset\Gamma_0\Subset \Gamma$ and the volume decay condition \eqref{eq:volume_decay} is convex.
\end{theorem}

Similarly as in \cite{Lan16}, the convexity estimate can be improved to a sharp estimate for the exscribed curvature in case the solution is exterior non-collapsing (see Theorem \ref{anc_extNC}).

We will also obtain optimal curvature pinching for $(m+1)$-convex, $m\in\{1,\dots,n-1\}$, ancient solutions of flows by concave or convex speeds and optimal inscribed curvature pinching for interior non-collapsing ancient solutions of flows by concave speeds (Proposition \ref{prop:m_convex_anc}). These will be used to prove further characterisations of the shrinking sphere for flows by concave speeds. The first weakens the convexity assumption $\Gamma=\Gamma_+$ of Theorem \ref{thm:sphere_char_typeI} to two-convexity:

\begin{theorem}\label{thm:2conv_sphere}
Let $f:\Gamma\to\R$ be a concave admissible speed such that $\Gamma_+\subset\Gamma\subset\Gamma_2$ and $f|_{\Gamma_+}$ is inverse-concave and let $X : M\times (-\infty, 1)\to \mathbb{R}^{n+1}$ be a compact, connected ancient solution of \eqref{eq:CF} satisfying $\kappa(M\times(-\infty,0])\subset\Gamma_0\Subset\Gamma$ and one of the conditions (\ref{hyp:typeI})--(\ref{hyp:bsr}) of Theorem \ref{thm:sphere_char_typeI}. Then $\{M_t\}_{t\in(-\infty,1)}$ is a shrinking sphere. 
\end{theorem}

For speeds which are strictly concave in non-radial directions, we can say more:
\begin{theorem}\label{thm:strict_conv_sphere}
Let $f:\Gamma\to\R$ be an admissible speed such that $\Gamma_+\subset\Gamma$ and $f$ is strictly concave in non-radial directions and let $X:M\times (-\infty, 1)\to \mathbb{R}^{n+1}$ be an ancient solution of \eqref{eq:CF} satisfying $\kappa(M\times(-\infty,0])\subset\Gamma_0\Subset\Gamma$ and one of the conditions (\ref{hyp:typeI})--(\ref{hyp:bsr}) of Theorem \ref{thm:sphere_char_typeI}. Then $\{M_t\}_{t\in(-\infty,1)}$ is a union of shrinking spheres. 
\end{theorem}

Our final result, which appears to be new also for the mean curvature flow, makes use of recent results for two-convex translating solutions of \eqref{eq:CF} \cite{Ha15,BoLa16}.
\begin{theorem}
\label{thm:grad_sphere}
There exists a positive constant $\delta_0 = \delta_0(n, f, \Gamma_0)$ with the following property: Let $f:\Gamma\to\R$ be a concave admissible speed such that $\Gamma_+\subset\Gamma\subset\Gamma_2$ and $f|_{\Gamma_+}$ is inverse-concave and let $X : M\times (-\infty, 1)\to \mathbb{R}^{n+1}$, $n\geq 3$, be an ancient solution of \eqref{eq:CF} satisfying $\kappa(M\times(-\infty,0])\subset\Gamma_0\Subset\Gamma$ and the gradient estimate
\begin{equation}
\label{eq:strong_grad}
\limsup_{t \to -\infty} \max_{M_t} \frac{|\nabla A|^2}{F^4}<\delta_0\,.
\end{equation}
Then $X$ is a shrinking sphere. 
\end{theorem}
\begin{remark}\mbox{}
The proof of Theorem \ref{thm:grad_sphere} works also in case $n=2$ if we assume, in addition, that the solution is interior non-collapsing (cf. \cite{HaHe,BoLa16}).
\end{remark}

\section{Preliminaries}\label{sec_prelims}

Let $X:M\times I\to\R^{n+1}$ be a solution of \eqref{eq:CF}. First, we recall that\footnote{Unless otherwise specified, integrals will be taken over all of $M$ and with respect to the measure $\mu_t$ induced by the immersion $X_t$.}
\begin{equation}
\label{area_decay}
\frac{d}{dt} \int \varphi  = \int \partial_t\varphi-\int \varphi FH
\end{equation}
for almost every $t$ for any test function $\varphi:M \times I \to \mathbb{R}$ for which the integrals are defined. 

Denote by $\sym(n)$ the space of symmetric $n\times n$ matrices. We will occasionally abuse notation by writing $Z \in \Gamma$ if the eigenvalue $n$-tuple of $Z \in \sym(n)$ is contained in a symmetric set $\Gamma\subset \R^n$. For a smooth, symmetric $g:\Gamma \to \mathbb{R}$ giving rise (by abuse of notation) to a function $g:\sym(n) \to \mathbb{R}$, vectors $v \in \mathbb{R}^n$, $z \in \Gamma$, and matrices $V \in \sym(n)$, $Z \in \Gamma$, we write
\begin{align*}
\dot g^i(z)v_i &= \frac{d}{ds}\Big|_{s=0} g(z + sv), \qquad \quad\;\;\;\; \ddot g^{ij} (z)v_i v_j = \frac{d^2}{ds^2}\Big|_{s=0} g(z + sv)\,,\\
\dot g^{ij}(Z)V_{ij} &=\frac{d}{ds}\Big|_{s=0} g(Z + s V), \quad \ddot g^{ij, kl}(Z)V_{ij}V_{kl} = \frac{d^2}{ds^2} \Big|_{s=0} g(Z+sV)\,.
\end{align*}
When $z$ is the eigenvalue $n$-tuple of $Z$, $\dot g^{i}(z)\delta^{ij}=\dot g^{ij}(Z)$. If, in addition, the components of $z$ are all distinct, then \cite{Gl,Ge96,An07}
\begin{equation}
\label{symm_secs}
\ddot g^{ij, kl} (Z) V_{ij} V_{kl} = \ddot g^{ij}(z) V_{ii} V_{jj} + 2\sum_{i > j} \frac{\dot g^i(z) - \dot g^j(z)}{z_i - z_j} (V_{ij})^2\,.
\end{equation}
Clearly, if $g$ is concave (resp. convex) with respect to the matrix components, then it is also concave (resp. convex) with respect to the eigenvalues. Conversely, if $g$ is concave (resp. convex) with respect to the eigenvalues, then it is also Schur concave (resp. convex) with respect to the eigenvalues, and hence concave (resp. convex) with respect to the matrix components \cite{EckHui89}.

To simplify notation, if $G:=g(\kappa)=g(A)$, then we write $\dot G^i=\dot g^i(\kappa)$ and similarly for higher derivatives.

We will find it convenient to define (in any orthonormal frame) $\langle u, v \rangle_F :=  \dot F^{ij}u_iv_j$ and  $|v|^2_F=\langle v, v\rangle_F$ for any $u$, $v \in TM$, as well as $\Delta_F := \dot F^{ij} \nabla_i \nabla_j$ (note however that, in general, $\Delta_F$ may not have a divergence structure). Uniform ellipticity $\kappa(M\times[0,T))\subset \Gamma_0\Subset\Gamma$ of the flow \eqref{eq:CF} ensures that $\langle\cdot\,,\,\cdot\rangle_F$ is uniformly equivalent to the induced metric and that $\Delta_F$ is uniformly elliptic. We make frequent use of the following Lemma (see \cite{An94a}):
\begin{lemma}
\label{curvature_evol}
Let $g$ be a smooth, symmetric, 1-homogeneous function defined on an open, symmetric cone $\Gamma\subset\R^n$. If $\kappa(x,t)\in \Gamma$ for every $(x,t) \in M \times I$ and $G(x,t):= g(\kappa(x,t))$ then, in any orthonormal frame,
\[(\partial_t - \Delta_F) G = |A|^2_F G + Q_{G,F}(\nabla A)\,,\]
where 
\[Q_{G,F}(\nabla A) := (\dot G^{kl} \ddot F^{pq,rs} - \dot F^{kl} \ddot G^{pq,rs}) \nabla_k A_{pq} \nabla_l A_{rs}\,.\]
In particular,
\[(\partial_t - \Delta_F) F = |A|^2_F F\,.\]
\end{lemma} 

We also make use of the following evolution inequalities for the inscribed and exscribed curvatures, which hold in the barrier sense (see \cite{ALM13} and \cite{NC2}): If $f$ is concave, then the inscribed curvature satisfies 
\begin{equation}
\label{insc_evol}
(\partial_t - \Delta_F) \overline k \leq  |A|^2_F \overline k - 2 \langle \nabla \overline k ,  \mathcal S(\nabla \overline k )\rangle _F
\end{equation}
on the set $\oversl  U := \{(x,t) \in M \times I : \overline k (x,t) > \kappa_n(x,t)\}$, where $ \mathcal S := (\overline k I - A)^{-1}$. 

If $f$ is convex, then the exscribed curvature satisfies
\begin{equation}
\label{exsc_evol}
(\partial_t - \Delta_F)\underline k \geq |A|^2_F\underline k + 2 \langle \nabla \underline k , \mathcal T (\nabla \underline k)\rangle_F
\end{equation}
on the set $\undersl U := \{(x,t) \in M\times [0, T) : \underline k (x,t) < \kappa_1(x,t)\}$, where $\mathcal T := (A-\underline k I)^{-1}$. 


A key ingredient in our proof of the cylindrical estimates will be a `Poincar\'{e}-type inequality' which is similar in spirit to those used in previous Stampacchia iteration arguments (cf. \cite{Huisk84,Cw85,Cw87,HuSi99a,HuSi99b,HuBre}). In order to state it, we define
\[\Cyl_m := \{(\underbrace{0,\dots, 0}_{m-\text{times}}, k, \dots, k) \in \mathbb{R}^n : k>0\}\]
and
\[\Cyl := \bigcup_{m =0}^{n-1}\Cyl_m.\]
\begin{lemma}[\cite{Lan16}]\label{lem:Poincare}
Let $\Gamma\subset\R^{n+1}$ be an open symmetric cone and $\Gamma_0$ an open, symmetric cone satisfying $\Gamma_0 \Subset (\Gamma\setminus \Cyl)$. Then there is a positive constant $\gamma(\Gamma_0, n)$ with the following property: Let $X:M\to \mathbb{R}^{n+1}$ be a smooth, closed hypersurface and $u \in W^{1,2}(M)$ a function satisfying $\kappa(\supp u) \subset \Gamma_0$. Then, for every $r>0$,
\[\gamma \int_M u^2 |A|^2 \leq r^{-1} \int_M |\nabla u|^2 + (1+r) \int_M u^2 \frac{|\nabla A|^2}{H^2}\,.\]
\end{lemma}

We will require a similar estimate when studying the inscribed curvature. This will be derived from the identity (see \cite{Lan16})
\begin{equation}
\label{insc_ident}
\frac{1}{2}H \leq \Div (\mathcal S^2 (\nabla \overline k) ) - \left\langle \mathcal S, \nabla_{\mathcal S^2 (\nabla \overline k)}A\right\rangle + \frac{1}{2} |\mathcal S(\nabla \overline k)|^2 \tr(\mathcal S)\; \text{ in }\; \oversl U\,,
\end{equation}
which holds in the distributional sense.

Finally, we make note of Andrews' differential Harnack inequality \cite{An94b} (cf. \cite{Cw91,Ham95}), which states that any strictly convex solution $X : M\times [t_0,T) \to \mathbb{R}^{n+1}$ of \eqref{eq:CF} moving by a convex or inverse-concave admissible speed satisfies 
\begin{equation}
\label{eq:Harnack}
\partial_t F-A^{-1}(\nabla F,\nabla F)+\frac{F}{2(t-t_0)} \geq 0\,.
\end{equation}

\section{One-sided curvature pinching}

Let $\varphi :\mathbb{R} \to [0,\infty)$ be a smooth, convex function that is positive on $(-\infty,0)$ and vanishes identically on $[0, \infty)$. Note that such a function necessarily satisfies $\varphi'\leq 0$. For concreteness, observe that the following function satisfies these properties:
\[
\varphi(r):=\begin{cases}
r^4\mathrm{e}^{-\frac{1}{r^2}}&\text{if}\; r<0\\
0&\text{if}\; r\geq 0\,.
\end{cases}
\] 
We define a symmetric, one-homogeneous function on $\Gamma$ by
\[g_1(z) := f(z)\sum_{i=1}^n \varphi\left(\frac{c_m f(z)-z_i}{f(z)}\right)\]
and set $G_1(x,t) = g_1(\kappa(x,t))$. This is essentially the pinching function we want to study. Observe that $G_1(x,t) = 0$ precisely when $\kappa_n(x,t) \leq c_m F(x,t)$. 

\begin{lemma}
\label{pinching_evol}
The function $G_1$ satisfies the differential equality 
\[(\partial_t - \Delta_F)G_1 \leq |A|^2_F G_1\]
on $M \times [0,T)$. 
\end{lemma}

\begin{remark}
Although we have chosen to work with a smooth function for $G_1$ here, it is possible to proceed using $G_1=\kappa_n-c_mF$ directly, since (although it is not smooth) it satisfies the differential inequality of Lemma \ref{pinching_evol} in the distributional sense.
\end{remark}

\begin{proof}[Proof of Lemma \ref{pinching_evol}]
By Lemma \ref{curvature_evol} it suffices to show that 
\[(\dot g_1^{kl}\ddot f^{pq,rs} - \dot f^{kl}\ddot g_1^{pq,rs})\big|_Z T_{kpq}T_{lrs} \leq 0\]
for every $Z \in \sym(n)$ and totally symmetric $T \in \mathbb{R}^n \otimes \mathbb{R}^n \otimes \mathbb{R}^n$. By continuity, we may assume the eigenvalues $z_i$ of $Z$ are all distinct. Suppressing dependencies on $Z$, we use \eqref{symm_secs} to write {\small
\begin{align*}
Q_{g,f}(T)&:=(\dot g_1^{kl}\ddot f^{pq,rs}-\dot f^{kl} \ddot g_1^{pq,rs})T_{kpq} T_{lrs} \\
&=  ( \dot g_1^k \ddot f^{pq}-\dot f^{k} \ddot g_1^{pq}) T_{kpp} T_{lqq}+ 2 \sum_{k\geq 1\atop p >q} \left(\dot g_1^k \frac{\dot f^p - \dot f^q}{\lambda_p - \lambda_q} - \dot f^k \frac{\dot g_1^p- \dot g_1^q}{\lambda_p - \lambda_q} \right)(T_{kpq})^2\\
&=: \text{I} + \text{II}.
\end{align*}}
Abbreviating $\xi_i(z) = \frac{c_m f(z) - z_i}{f(z)}$, we compute
\begin{align*}
\dot g_1^k = -\varphi'(\xi_k)+\dot f^k \sum_i \left(\varphi(\xi_i)+ \frac{z_i}{f} \varphi'(\xi_i)\right)
\end{align*}
and
\begin{align*}
\ddot g_1^{kl} &= \frac{1}{f}\sum_i \varphi''(\xi_i)\left(\frac{\lambda_i}{f}\dot f^l - \delta^l_i \right)\left(\frac{\lambda_i}{f} \dot f^k-\delta^k_i\right)\\
&\;\;\;\; +\ddot f^{kl} \sum_i \left( \varphi(\xi_i)+ \frac{\lambda_i}{f} \varphi'(\xi_i)\right),	
\end{align*}
from which we obtain
\begin{align*}
\text{I} = -\sum_k \varphi'(\xi_k)\ddot f^{pq}T_{kpp}T_{kqq} - \frac{f^k}{f}\sum_i \varphi''(\xi_i)\left(\frac{\lambda_i}{f}\dot f^p T_{kpp} - T_{kii}\right)^2
\end{align*}
and
\begin{align*}
\text{II}=-\sum_k \varphi'\left(\xi_k \right)\sum_{p>q} \frac{\dot f^p - \dot f^q}{\lambda_p - \lambda_q}(T_{kpq})^2 +  \dot f^k \sum_{p>q} \frac{\varphi'\left(\xi_p\right)-\varphi'\left(\xi_q\right)}{\lambda_p - \lambda_q}(T_{kpq})^2.
\end{align*}
By the concavity and monotonicity properties of $f$ and $\varphi$, we see that both terms are non-positive (note that $z_p\geq z_q\Rightarrow \xi_p\leq \xi_q$).
\end{proof}

We now define $G_2 := 2\Theta F - H- |A|$ on $M\times [0,T)$, where 
\[\Theta:= \max\left \{\frac{z_1 + \dots + z_n + |z|}{f(z)} : z \in \Gamma_0 \right \}\,.\]
 Then $2\Theta F\geq G_2 \geq \Theta F>0$. Rather than working with $G_1$ directly, we will study the modification $G:=G_1^2/G_2$, since it enjoys a better evolution equation; namely, it provides a good gradient of curvature term which we will need later. Note that $G$ is still a smooth, symmetric, non-negative, one-homogeneous function of the principal curvatures which vanishes precisely where $\kappa_n\leq c_mF$.

\begin{lemma}\label{lem:cylindrical_goodgrad}
There is a constant $\gamma=\gamma(n,f,\Gamma_0,\varepsilon)>0$ such that 
\begin{equation}
\label{cylindrical_aux_evol}
(\partial_t-\Delta_F)G_2 \geq |A|^2_FG_2+\gamma\frac{|\nabla A|^2}{F}
\end{equation}
wherever $G>\varepsilon F$. Consequently, there is a (possibly different) constant $\gamma=\gamma(n,f,\Gamma_0,\varepsilon)>0$ such that
\begin{equation}
\label{mod_pinching_evol}
(\partial_t - \Delta_F)G \leq |A|^2_FG-\gamma G \frac{|\nabla A|^2}{F^2}
\end{equation}
wherever $G>\varepsilon F$.
\end{lemma}
\begin{proof}
First, we compute  
\begin{align*}
 (\partial_t-\Delta_F) G &= 2\frac{G_1}{G_2}(\partial_t-\Delta_F)G_1 - \frac{G}{G_2}(\partial_t - \Delta_F)G_2\\
 {}&\;\;\;\; - \frac{2}{G_2} \left| \nabla G_1 - \frac{G_1}{G_2} \nabla G_2\right|_F^2\,.
 \end{align*}
By Lemmas \ref{curvature_evol} and \ref{pinching_evol} (and since $G_2 \leq 2\Theta F$) it suffices to show that
\begin{equation}
\label{quadratic_G_2}
Q_{g_2,f}\big|_Z(T) \geq \gamma\frac{|T|^2}{F(Z)}
\end{equation}
for every diagonal $Z \in \sym(n)$ and totally symmetric $T \in \mathbb{R}^n \otimes \mathbb{R}^n \otimes \mathbb{R}^n$. Denote $N(Z):=|Z|+\tr(Z)$. Then $N$ is monotone non-decreasing and convex and hence
\begin{align*}
Q_{g_2, g}\big|_Z(T) &=(\dot g_2^{kl}\ddot f^{pq,rs} - \dot f^{kl}\ddot g_2^{pq,rs})\big|_Z T_{kpq}T_{lrs}\\
&= (-\dot N^{kl}\ddot f^{pq,rs} + \dot f^{kl}\ddot N^{pq,rs})\big|_Z T_{kpq}T_{lrs}\\
&\geq 0\,.
\end{align*}
Suppose that equality occurs. Since $N$ is strictly convex in non-radial directions, each $T_{kpq} $ must then be of the form $a_k Z_{pq}$ for some $a_k \in \mathbb{R}$, so symmetry implies that $T_{kpq} \not = 0$ only when $k=p=q$. Since we can assume $T_{kkk}\neq 0$ for some $k$, this leads to the contradiction
\[T_{kkk} = a_k Z_{kk} = \frac{Z_{kk}}{Z_{ll}}a_k Z_{ll} = \frac{Z_{kk}}{Z_{ll}} T_{kll}=0\]
whenever $k$ and $l$ are distinct indices with $Z_{kk}$ and $Z_{ll}$ both non-zero. This means $Z$ can have at most one non-zero entry. For $m\leq n-2$, this contradicts $Z\in \Gamma_0\Subset \Gamma_{m+1}$. For $m=n-1$, it contradicts $g(z)\geq \varepsilon f(z)$. It follows that $Q_{g_2, f}\big|_Z(T)$ attains a positive minimum on the compact set $\{(Z,T) \in \Gamma_{0} \times \mathbb{R}^n\odot \mathbb{R}^n \odot \mathbb{R}^n: |Z|, |T| = 1\}$, which we set equal to $\gamma$. The claim then follows from the homogeneity of $Q_{g_2, f}$ in $Z$ and $T$.
\end{proof}

\begin{proof}[Proof of Theorem \ref{cylindrical_ests}]

We will use Stampacchia iteration to bound the function $G_{\sigma} := (G - \varepsilon F)F^{\sigma-1}$ for some $\sigma \in (0, \frac{1}{2})$ and any $\varepsilon >0$. This suffices to prove the theorem: Fix $\eta>0$ and suppose that $\kappa_n-c_m F \geq \eta F$. If $G_{\sigma} \leq C$ then, by the convexity and monotonicity properties of $\varphi$, we can estimate
\begin{align*}
\kappa_n-c_m F \leq \frac{\eta}{\varphi(-\eta)}G_1\leq \frac{2\Theta \eta}{\varphi(-\eta)^2} G \leq\frac{2\Theta \eta}{\varphi(-\eta)^2}  (\varepsilon F + CF^{1-\sigma})\,.
\end{align*}
Choosing $\varepsilon = \varepsilon (n ,f, \Gamma_0, \eta)$ small enough and applying Young's inequality, we then obtain
\[\kappa_n - c_m F \leq \eta F + C_\eta\]
as required.  

The first step is to establish an $L^p$ estimate for $G_{\sigma, +} := \max\{G_\sigma, 0\}$. 
\begin{proposition}
\label{cyl_Lp}
There is a constant $\ell=\ell(n,f,\Gamma_0,\varepsilon)>0$ such that
\[\frac{d}{dt} \int G^p_{\sigma, +}\,d\mu \leq 0
\]
for $p\geq \ell^{-1}$ and $\sigma\leq \ell p^{-\frac{1}{2}}$. 
\end{proposition} 

\begin{proof}
Lemma \ref{lem:cylindrical_goodgrad} provides us with a constant $\gamma(n,f,\Gamma_0,\varepsilon)>0$ such that
\begin{align*}
(\partial_t - \Delta_F) G_{\sigma} \leq{}& \sigma |A|_F^2 G_{\sigma} -2\gamma GF^{\sigma - 1}\frac{|\nabla A|^2}{F^2}  \notag\\
& + \frac{2(1-\sigma)}{F} \langle \nabla G_{\sigma} , \nabla F\rangle_F- \frac{\sigma(1-\sigma)}{F^2}|\nabla F|^2_F  \notag\\
\leq{}&\sigma |A|_F^2G_{\sigma}  -2\gamma G_{\sigma}\frac{|\nabla A|^2}{F^2}+ \frac{2(1-\sigma)}{F} \langle \nabla G_{\sigma} , \nabla F\rangle_F .
\end{align*}
Combined with \eqref{area_decay}, this allows us to estimate
\begin{align}\label{eq:cylindrical_Lp_evolve}
\frac{d}{dt}\int G_{\sigma,+}^p \leq{}& p \int G_{\sigma,+}^{p-1} \Delta_FG_{\sigma}+ \sigma p \int G_{\sigma,+}^p |A|^2_F-2\gamma_1 p \int  G_{\sigma, +}^p \frac{|\nabla A|^2}{F^2}\nonumber\\
{}& + 2(1-\sigma)p\int  G_{\sigma,+}^{p-1}F^{-1} \langle \nabla G_{\sigma}, \nabla F \rangle_F-\int G_{\sigma,+}^p HF.
\end{align}
Since $\kappa_1+\dots+\kappa_{m+1}>0$, the final term is non-positive and can be dropped. Integrating by parts and using Young's inequality, we may estimate, for $p>2$,
\begin{align}\label{eq:integration_by_parts}
p\int G_{\sigma,+}^{p-1} \Delta_F G_{\sigma} \leq{}& - p(p-1)\int \dot F^{ij}G_{\sigma,+}^{p-2} \nabla_j G_{\sigma}\nabla_i G_{\sigma}\notag \\
{}& -p\int G_{\sigma, +}^{p-1}\ddot F^{ij, kl}\nabla_j A_{kl} \nabla_i G_{\sigma} \notag\\
\leq{}& -(p(p-1) - Cp^\frac{3}{2})\int G_{\sigma,+}^{p-2} |\nabla G_{\sigma}|_F^2\notag\\
{}& + Cp^\frac{1}{2} \int G_{\sigma,+}^p \frac{|\nabla A|^2}{F^2}
\end{align}
as long as the constant $C<\infty$ satisfies
\[2C \geq \max \{f(z)|\ddot f|(z): z \in \Gamma_0\}\,.\]
Since, in any orthonormal frame,
\[|\nabla F|^2_F = \dot f^k \dot f^p \dot f^q \nabla_k A_{pp} \nabla_k A_{qq}\,,\]
we may estimate the inner product term by
\begin{align}
\label{eq:inner_product}
2(1-\sigma)p\int G_{\sigma, +}^p\left\langle \frac{\nabla G_{\sigma}}{G_{\sigma}}, \frac{\nabla F}{F} \right\rangle_F  &\leq p^\frac{3}{2} \int G_{\sigma, +}^{p-2}|\nabla G_{\sigma}|_F^2\nonumber\\
{}&\;\;\;\;+Cp^\frac{1}{2}\int G_{\sigma, +}^p \frac{|\nabla A|^2}{F^2}
\end{align}
wherever $G_{\sigma} >0$ as long as $C$ also satisfies
\[
C^{\frac{1}{3}}\geq\max\{ |\dot f|(z)  : z \in \Gamma_{0}\}\,.
\]
Combining \eqref{eq:cylindrical_Lp_evolve}, \eqref{eq:integration_by_parts} and \eqref{eq:inner_product}, we obtain
\begin{align*}
\frac{d}{dt}\int G_{\sigma,+}^p &= -(p(p-1)-(C+1)p^\frac{3}{2})\int G_{\sigma,+}^{p-2} |\nabla  G_{\sigma}|_F^2\\
&\;\;\;\; - 2(\gamma p - Cp^\frac{1}{2}) \int G_{\sigma,+}^p \frac{|\nabla A|^2}{F^2}+\sigma p \int G_{\sigma,+}^p |A|_F^2\,.
\end{align*}
Thus, assuming $\gamma\leq\frac{1}{2}$, we can estimate
\begin{align}
\label{cyl_Lp_prel}
\frac{d}{dt} \int G_{\sigma,+}^p &\leq -\gamma p^2\int G_{\sigma,+}^{p-2} |\nabla  G_{\sigma}|_F^2 \notag - \gamma p \int G_{\sigma,+}^p \frac{|\nabla A|^2}{F^2}\notag \\
&\;\;\;\; +\sigma p \int G_{\sigma,+}^p |A|_F^2
\end{align}
for $p\geq \ell^{-1}(\gamma,C)=\ell^{-1}(n,f,\Gamma_0,\varepsilon)$ sufficiently large.

Next, observe that $\Gamma_{m+1} \cap \Cyl_i $ is empty for every $m+1 \leq i \leq n$. On the other hand, if $\kappa (x,t) \in \Cyl_i$ for some $0\leq i\leq m$ then $\kappa_n(x,t) \leq c_m F(x,t)$. We conclude that the support of $G_{\sigma, +}$ is compactly contained away from $\Cyl$, at a normalized distance dependent on $\varepsilon$. This allows us to apply the Poincar\'{e} inequality, Lemma \ref{lem:Poincare}, with $u^2 = G_{\sigma,+}^p$ and $r = p^\frac{1}{2}$. Estimating $F\leq C(n,f,\Gamma_0)H$, and assuming $\gamma=\gamma(n,f,\Gamma_0,\varepsilon)>0$ is sufficiently small, we obtain
\begin{align}
\label{cylindrical_poinc}
\gamma \int G_{\sigma, +}^p |A|^2 \leq p^\frac{3}{2}\int G_{\sigma,+}^{p-2}|\nabla G_{\sigma}|^2+p^\frac{1}{2}\int G_{\sigma,+}^p\frac{|\nabla A|^2}{F^2}\,.
\end{align}
Substituting \eqref{cylindrical_poinc} into \eqref{cyl_Lp_prel} yields
\begin{align}\label{cyl_Lp_ancient}
\frac{d}{dt} \int G_{\sigma,+}^p &\leq -\left(\gamma^2p^\frac{1}{2}-\sigma p\right) \int |A|_F^2 G_{\sigma, +}^p\,.
\end{align}
for $p\geq \ell^{-1}(n,f,\Gamma_0,\varepsilon)$. The claim follows.
\end{proof}
The Stampacchia iteration argument leading to an upper bound for $G_{\sigma}$ for some $\sigma>0$ now proceeds as in \cite{Huisk84} (see also \cite[Section 5]{ALM14}, where the argument is applied to one-homogeneous fully nonlinear speeds).

\end{proof}

\section{Inscribed curvature pinching}
\label{sec_insc}

In this section, we apply the the cylindrical estimates \eqref{cylindrical_ests} to prove Theorem \ref{inscribed_improves}. We will first prove the estimate for $m\leq n-2$ since the case $m=n-1$ is more subtle (although the proof in the latter case also works for $m\leq n-2$).

\subsection{Case 1: $m\leq n-2$}

We first set $G_1 := \max\{\overline k - c_m F,0\}$ and $G:=G_1^2/G_2$, where again $G_2 := 2\Theta F-H -|A|$ with $\Theta=\Theta(n,f,\Gamma_0)$ chosen so that $\Theta F \geq H + |A|$ on $M \times [0,T)$. We also impose the condition $\Theta \geq \Lambda$, so that $G_1 \leq G_2$. The inequalities \eqref{insc_evol} and \eqref{cylindrical_aux_evol} then imply that
\begin{align*}
 (\partial_t-\Delta_F) G &\leq  2\frac{G}{G_1}(\partial_t-\Delta_F)G_1 - \frac{G}{G_2}(\partial_t - \Delta_F)G_2 \notag \\
 &\leq |A|_F^2 G - 4\frac{G}{G_1}\langle \nabla \overline k , \mathcal S(\nabla \overline k)\rangle_F   -\gamma \frac{G}{G_2}\frac{|\nabla A|^2}{F}
\end{align*}
distributionally on the set $\oversl U := \{ (x,t) \in M \times (0,T) : \overline k (x,t) > \kappa_n (x,t)\}$. Since $\bar k - \kappa_i \leq n\bar k - H \leq n\Lambda F$ we can estimate the eigenvalues of $\mathcal S = (\overline k I- A)^{-1}$ from below to obtain
\begin{equation}
\label{insc_G_evol}(\partial_t - \Delta_F)G \leq |A|^2G - \gamma G\left(\frac{|\nabla \overline k|^2}{F^2} + \frac{|\nabla A|^2}{F^2} \right),
\end{equation}
where $\gamma>0$ may now depend on $n$, $f$, $\Gamma_0$ and $\Lambda$. 

Fix any $\varepsilon >0$. Since $\min_{M\times\{t\}}F$ is non-decreasing, Theorem \eqref{cylindrical_ests} allows us to estimate
\begin{equation}\label{usingknest}
\kappa_n - c_m F \leq \frac{\varepsilon}{2} F + K \min\{1, F\}
\end{equation}
for some $K < \infty$ depending only on $n$, $f$, $\Gamma_0$, $M_0$ and $\varepsilon$. We will use Stampacchia iteration to show that, for some small $\sigma>0$, the function
\begin{equation}
\label{insc_pinching}
G_{\sigma} := (G-\varepsilon F) F^{\sigma - 1} -K
\end{equation}
can be bounded purely in terms of $n$, $f$, $\Gamma_0$, $\varepsilon$ and the initial data, thereby proving the theorem. Since $G_1 \leq G_2$, \eqref{usingknest} yields (cf. \cite{Brend15})
\begin{align*}
\overline k - \kappa_n &\geq G_1 - \frac{\varepsilon}{2}F - K \min\{1,F\}\\
		&\geq G - \frac{\varepsilon}{2} F - K \min\{1, F\}\\
		&\geq F^{1-\sigma}(G_{\sigma} + K)+\frac{\varepsilon}{2}F - K \min\{1, F\}\\
		&\geq G_\sigma F^{1-\sigma} + \frac{\varepsilon}{2}F
\end{align*}
so that, whenever $G_{\sigma} \geq 0$,
\[\overline k - \kappa_n \geq \frac{\varepsilon}{2}F.\]
Thus, $G_{\sigma,+}:=\max\{G_\sigma,0\}$ is supported in $\overline U$, and hence
\begin{align*}
(\partial_t - \Delta_F) G_\sigma &\leq \sigma |A|^2_F (G_{\sigma} +K)- \gamma G_{\sigma}\left(\frac{|\nabla \overline k|^2}{F^2} + \frac{|\nabla A|^2}{F^2}\right)\\
&\;\;\;\; + 2(1-\sigma)\left\langle \nabla G_{\sigma}, \frac{\nabla F}{F}\right\rangle_F
\end{align*}
distributionally in $\supp(G_{\sigma,+})$. 

As for the cylindrical estimates, the first step in the iteration argument is to prove an $L^p$ estimate for $G_{\sigma,+}$.

\begin{proposition}
\label{insc_Lp_flags}
There is a constant $\ell=\ell(n,f,\Gamma_0,\Lambda,\varepsilon)$ such that
\[\frac{d}{dt} \int G_{\sigma, +}^p\leq
\sigma  K^p \int |A|_F^2 \]
for $p\geq \ell^{-1}$ and $\sigma\leq \ell p^{-\frac{1}{2}}$. 
\end{proposition}
\begin{proof}
The same arguments used to obtain \eqref{cyl_Lp_prel} provide us with a positive constant $a=a(n,f,\Gamma_0,\Lambda)$ such that
\begin{align}
\label{insc_Lp_prelim}
\frac{d}{dt} \int G_{\sigma,+}^p  &\leq -ap^2\int G_{\sigma, +}^{p-2} |\nabla G_{\sigma}|^2 -ap\int G_{\sigma, +}^p \frac{|\nabla A|^2}{F^2} \notag\\
&\;\;\;\; - ap\int G_{\sigma,+}^p \frac{|\nabla \overline k|^2}{F^2} +\sigma p\int G_{\sigma,+}^{p-1}|A|_F^2(G_{\sigma} +K) 
\end{align}
for $p$ sufficiently large. We then use Young's inequality to estimate the final term by
\begin{align*}
\sigma p \int G_{\sigma, +} ^{p-1} |A|_F^2 (G_{ \sigma}+K)\leq{}& \sigma \int |A|_F^2 (pG_{\sigma, +} ^p + K^p +(p-1)G^p_{ \sigma, +}) \\
\leq{}& 2\sigma p \int G_{\sigma, +} ^p |A|_F^2  + \sigma K^p \int |A|_F^2
\end{align*}
and use the inequality \eqref{insc_ident} to control the remaining bad term:
\begin{lemma}[Cf. \cite{Lan16}]
\label{insc_poinc}
There is a constant $\gamma=\gamma(n,f,\Gamma_0,\Lambda,\varepsilon)>0$ such that
\begin{align*}
\gamma \int G_{\sigma, +}^p|A|_F^2 \leq{}& p^\frac{3}{2} \int G_{\sigma, +}^{p-2} |\nabla G_{\sigma}|^2  +p^\frac{1}{2}\int G^p_{\sigma, +}\frac{|\nabla \overline k|^2}{F^2}+\int G_{\sigma, +}^p \frac{|\nabla A|^2}{F^2}
\end{align*}
for $p\geq 2$.
\end{lemma}
\begin{proof}
We fix $\gamma = \gamma(n,f,\Gamma_0)>0$ so that $2\gamma |A|_F^2 \leq FH$, and use \eqref{insc_ident} to estimate
{\small
\begin{align}
\label{eq:insc_poinc1}
\gamma \int G^{p}_{\sigma, +} |A|_F^2 \leq{}& \frac{1}{2} \int G_{\sigma, +}^p FH \notag \\
		\leq{}& \int G_{\sigma, +}^p F \Big(\Div (\mathcal S^2 \nabla \overline k ) - \langle \mathcal S, \nabla_{\mathcal S^2 \nabla \overline k}A\rangle +\frac{1}{2} |\mathcal S\nabla k|^2 \tr(\mathcal S)\Big)\,. 
\end{align}} \hspace{-1.5mm}
Wherever $G_{\sigma} \geq 0$, we can estimate $|\mathcal S|^2 \leq C(n,f,\Lambda,\varepsilon)F^{-2}$. Hence, integrating by parts,
\begin{align*}
\int G_{\sigma, +}^p F \Div (\mathcal S^2 \nabla \overline k ) &\leq -\int G_{\sigma, +}^p F^2 \left\langle p\frac{\nabla G_{\sigma}}{G_{\sigma}} + \frac{\nabla F}{F}, \frac{\mathcal S^2 \nabla \overline k}{F}\right\rangle\\
& \leq \int G_{\sigma, +}^p \left( p \frac{|\nabla G_\sigma|}{G_\sigma} + \frac{|\nabla F|}{F}\right)\frac{|\nabla \overline k|}{F}.
\end{align*}
Estimating the remaining terms on the right of \eqref{eq:insc_poinc1}, we obtain
\begin{align*}
\gamma \int G^{p}_{\sigma, +}   |A|_F^2 &\leq C\int  G_{\sigma,+}^p \left(p \frac{|\nabla G_{\sigma}|}{G_\sigma} \frac{|\nabla \overline k|}{F} + \frac{|\nabla F|}{F}\frac{|\nabla \overline k|}{F} \right.\\
&\hspace{4cm}\left. + \frac{|\nabla \overline k|}{F} \frac{|\nabla A|}{F} + \frac{|\nabla \overline k|^2}{F^2}\right).
\end{align*}
The result then follows from $|\nabla F| \leq C(n,f,\Gamma_0)|\nabla A|$ and Young's inequality.
\end{proof}
Combining Lemma \ref{insc_poinc} with \eqref{insc_Lp_prelim} yields (for $p$ sufficently large)
\begin{align}\label{insc_Lp_ancient1}
\frac{d}{dt} \int G_{ \sigma, +}^p  &\leq -(a\gamma p^\frac{1}{2}-2\sigma p)\int |A|^2_FG_{\sigma,+}^p +\sigma K^p\int |A|_F^2\,,
\end{align}
where $\gamma$ is the constant from Lemma \ref{insc_poinc}. The claim follows.
\end{proof}

We now estimate $|A|^2_F \leq C(n,f,\Gamma_0)HF$, so that
\begin{align}
\label{insc_unif_Lp}
\frac{d}{dt} \int (G_{\sigma, +}^p + \sigma K^p C) &\leq \sigma K^p \int |A|^2_F - \sigma K^p C\int HF \leq 0.
\end{align}
This yields a uniform bound for the $L^2$ norm of $G_{\sigma, +}^{\frac{p}{2}}$ in terms of $n$, $f$, $\Gamma_0$, $M_0$, $\varepsilon$, $\sigma$, and $p$. This suffices to apply the Stampacchia iteration argument.

\subsection{Case 2: $m=n-1$}

We again set $G:=G_1^2/G_2$, where $G_1:=\max\{\overline k-c_{n-1}F,0\}$ and $G_2:=2\Theta F-H-|A|$ with $\Theta=\Theta(n,f,\Gamma_0,\Lambda)$ chosen so that $G_2\geq \max\{G_1,\Theta F\}$. Then, just as in Case 1,
\begin{align*}
(\partial_t-\Delta_F)G\leq |A|^2_FG-4\frac{G}{G_1}\left\langle \mathcal{S}(\nabla\overline k),\nabla\overline k\right\rangle_F-\frac{G}{G_2}Q_{G_2,F}(\nabla A)
\end{align*}
distributionally on the set $\oversl U := \{ (x,t) \in M \times (0,T) : \overline k (x,y) > \kappa_n (x,t)\}$. Setting $G_\sigma:=(G-\varepsilon F)F^{\sigma-1}-K$ and $G_{\sigma,+}:=\max\{G_\sigma,0\}$ with $K$ chosen as before, this yields
{\small\begin{align*}
(\partial_t-\Delta_F)G_\sigma\leq{}& \sigma |A|^2_F(G_{\sigma}+K)-GF^{\sigma-1}\left(4\left\langle \mathcal{S}(\nabla\overline k),\frac{\nabla\overline k}{G_1}\right\rangle_F+\frac{Q_{G_2,F}(\nabla A)}{G_2}\right)\\
{}&+2(1-\sigma)\left\langle\nabla G_\sigma,\frac{\nabla F}{F}\right\rangle_F-\sigma(1-\sigma)(G_\sigma+K)\frac{|\nabla F|^2_F}{F^2}\\
\leq{}&\sigma |A|^2_F(G_{\sigma}+K)-4GF^{\sigma-1}\left\langle \mathcal{S}(\nabla\overline k),\frac{\nabla\overline k}{G_1}\right\rangle_F - G_\sigma  \frac{Q_{G_2,F}(\nabla A)}{G_2} \\
{}&+2(1-\sigma)G_\sigma\left\langle\frac{\nabla G_\sigma}{G_\sigma},\frac{\nabla F}{F}\right\rangle_F- \frac{\sigma}{2}G_\sigma\frac{|\nabla F|^2_F}{F^2}
\end{align*}}\hspace{-1.5mm}
distributionally in $\mathrm{spt}(G_{\sigma,+})\subset \oversl U$ for any $\sigma\in(0,\frac{1}{2})$. 

We will use the good final term to get a good gradient of curvature term:
\begin{lemma}
\label{lem:insc_goodgrad}
There is a constant $\gamma=\gamma(n,f,\Gamma_0)>0$ such that
\[
\frac{Q_{G_2,F}(\nabla A)}{G_2}+\frac{\sigma}{2}\frac{|\nabla F|^2_F}{F^2}\geq 4\gamma\sigma \frac{|\nabla A|^2}{F^2}
\]
for any $\sigma\in (0,1)$.
\end{lemma}
\begin{proof}
Since $G_2\leq 2\Theta F$, it suffices to prove that
\begin{align*}
\mathcal{Q}(Z,T):={}&\sigma^{-1}\Theta^{-1}f(Z)Q_{g_2,f}\big|_{Z}(T)+\dot f^{kl}|_{Z}\dot f^{pq}|_{Z}\dot f^{rs}|_{Z}T_{kpq}T_{lrs}\\
\geq{}& 4\gamma(n,f,\Gamma_0)>0
\end{align*}
for all $(Z,T)\in \{(Z,T)\in\Gamma_0\times(\R^n\odot\R^n\odot\R^n):|Z|,|T|=1\}$ and $\sigma\in(0,1)$. In fact, it suffices to prove this when $\sigma=1$ since, for $\sigma<1$,
\[
\min_{|Z|=|T|=1}\mathcal{Q}(Z,T)\geq \min_{|Z|=|T|=\sigma=1}\mathcal{Q}(Z,T)
\]
and the right hand side depends only on $n$, $f$ and $\Gamma_0$. 

As in the proof of Lemma \ref{lem:cylindrical_goodgrad}, the first term is non-negative and can only vanish if $T$ has exactly one non-zero component, $T_{nnn}$ say, in which case the second term is
\[
(\dot f^{nn}|_Z)^3|T_{nnn}|^2=(\dot f^{n}|_z)^3>0\,,
\]
where $z$ is the eigenvalue $n$-tuple of $Z$, so $\mathcal{Q}$ is strictly positive as required. 
\end{proof}

Estimating also
\[
\left\langle \mathcal{S}(\nabla\overline k),\frac{\nabla\overline k}{G_1}\right\rangle_F\geq \gamma(n,f,\Gamma_0,\Lambda)\frac{|\nabla \overline k|^2}{F^2}\,,
\]
we obtain
\begin{align*}
(\partial_t-\Delta_F)G_\sigma\leq{}&|A|^2_F(G_{\sigma}+K)-\gamma G_\sigma \frac{|\nabla \overline k|^2}{F^2}-4\gamma \sigma G_\sigma \frac{|\nabla A|^2}{F^2} \nonumber\\
{}&+2(1-\sigma)G_\sigma\left\langle\frac{\nabla G_\sigma}{G_\sigma},\frac{\nabla F}{F}\right\rangle_F\nonumber\\
\leq{}&|A|^2_F(G_{\sigma}+K) -\gamma G_\sigma \frac{|\nabla \overline k|^2}{F^2}-3\gamma \sigma G_\sigma \frac{|\nabla A|^2}{F^2}\nonumber\\
{}&+\frac{C}{\sigma} \frac{|\nabla G_\sigma|_F^2}{G_\sigma}\,,
\end{align*}
where $\gamma>0$ and $C<\infty$ are constants which depend only on $n$, $f$, $\Gamma_0$ and $\Lambda$. Henceforth, $\gamma>0$ will be fixed but $C<\infty$ may additionally depend on $\varepsilon$, and may change value from line to line.

So consider
\begin{align}\label{eq:insc_hard1}
\frac{d}{dt}\int G_{\sigma,+}^p &\leq p\int G_{\sigma,+}^{p-1} \Delta_F G_\sigma   -\gamma p\int G_{\sigma,+}^{p-1}\frac{|\nabla \overline k|^2}{F^2}\notag \\
&\;\;\;\;-3\gamma \sigma p\int G_{\sigma,+}^{p-1}\frac{|\nabla A|^2}{F^2} +C\sigma^{-1}p\int G_{\sigma, +}^{p-2}|\nabla G_\sigma |_F^2\notag \\
&\;\;\;\; + p\int G_{\sigma, +}^{p-1} |A|^2_F(G_\sigma + K)\,.
\end{align}
Integrating the diffusion term by parts and applying Young's inequality yields
\begin{align}\label{eq:insc_hard2}
p\int G_{\sigma,+}^{p-1} \Delta_F G_\sigma={}&-p\int G_{\sigma,+}^{p}\left((p-1)\frac{|\nabla G_\sigma|_F^2}{G_\sigma^2}+\ddot F^{kl,pq}\nabla_kA_{pq}\frac{\nabla_lG_\sigma}{G_\sigma}\right)\nonumber\\
\leq{}&-p\int G_{\sigma,+}^{p}\left((p-1)\frac{|\nabla G_\sigma|_F^2}{G_\sigma^2}-C\frac{|\nabla A|}{F}\frac{|\nabla G_\sigma|}{G_\sigma}\right)\nonumber\\
\leq{}&-p\int G_{\sigma,+}^{p}\left((p-1-C\sigma^{-1})\frac{|\nabla G_\sigma|_F^2}{G_\sigma^2}-\gamma\sigma\frac{|\nabla A|^2}{F^2}\right)\,.
\end{align}
The remaining bad term can be estimated as before by
\begin{align}\label{eq:insc_hard3}
\sigma p\int G_{\sigma,+}^{p-1}|A|^2_F(G_{\sigma}+K)\leq 2\sigma p\int |A|^2_F G_{\sigma,+}^{p}+\sigma K^p\int|A|^2_F.
\end{align}
We recall from the proof of Lemma \ref{insc_poinc} that
\begin{align*}
\int |A|_F^2 G^{p}_{\sigma, +}  &\leq C\int  G_{\sigma,+}^p \left(p \frac{|\nabla G_{\sigma}|}{G_\sigma} \frac{|\nabla \overline k|}{F} + \frac{|\nabla F|}{F}\frac{|\nabla \overline k|}{F} \right.\\
&\hspace{4cm}\left. + \frac{|\nabla \overline k|}{F} \frac{|\nabla A|}{F} + \frac{|\nabla \overline k|^2}{F^2}\right).
\end{align*}
Young's inequality then implies
\begin{align}
\label{eq:insc_hard4}
3\sigma p\int  |A|_F^2G^{p}_{\sigma, +} &\leq C\sigma p\int  G_{\sigma,+}^p \left(p^\frac{3}{2} \frac{|\nabla G_{\sigma}|^2}{G_\sigma^2}+ (1+p^\frac{1}{2}) \frac{|\nabla \overline k|^2}{F^2}\right)\notag \\
&\;\;\;\; + \gamma \sigma p\int G_{\sigma,+}^p\frac{|\nabla A|^2}{F^2} .
\end{align}

Putting \eqref{eq:insc_hard1}, \eqref{eq:insc_hard2}, \eqref{eq:insc_hard3} and \eqref{eq:insc_hard4} together, we obtain
\begin{align*}
\frac{d}{dt}\int G_{\sigma,+}^p\leq{}& -p\left(p-1-C\sigma^{-1}-C\sigma p^{\frac{3}{2}}\right)\int G_{\sigma,+}^{p-2}|\nabla G_\sigma|^2_F\\
{}&-p\left(\gamma-C\sigma (p^{\frac{1}{2}}+1)\right)\int G_{\sigma,+}^{p}\frac{|\nabla \overline k|^2}{F^2}\\
&-\gamma\sigma p\int G_{\sigma,+}^{p}\frac{|\nabla A|^2}{F^2}\\
&+\sigma K^p\int |A|^2_F - \sigma p \int |A|^2_F G_{\sigma, +}^p\,,
\end{align*}
Recalling \eqref{eq:insc_hard4}, we obtain, for $p\gg 1$ and $p^{-1}\lesssim \sigma\lesssim p^{-\frac{1}{2}}$,
\begin{align}\label{insc_Lp_ancient2}
\frac{d}{dt}\int G_{\sigma,+}^p\leq{}& \sigma K^p\int |A|^2_F - \sigma p \int |A|^2_F G_{\sigma, +}^p\,.
\end{align}
Estimating $|A|^2_F \leq C(n,f,\Gamma_0)HF$, this yields
\begin{align*}
\frac{d}{dt} \int (G_{\sigma, +}^p + \sigma K^pC) \leq -\sigma p\int |A|^2_F G_{\sigma,+}^p\,.
\end{align*}

In summary, we have proven the following proposition:
\begin{proposition}\label{prop:insc_hard}
There exist constants $\ell=\ell(n,f,\Gamma_0,\Lambda,\varepsilon)>0$ and $C=C(n,f,\Gamma_0)<\infty$ such that
\begin{align*}
\frac{d}{dt} \int (G_{\sigma, +}^p + \sigma K^pC) \leq{}&-\sigma p\int|A|^2_F G_{\sigma,+}^p
\end{align*}
for almost every $t\in (0,T)$ whenever $p\geq \ell^{-4}$ and $\ell^{-1}p^{-1}\leq \sigma\leq \ell p^{-\frac{1}{2}}$.
\end{proposition}
Integrating, we obtain a uniform bound for the $L^2$ norm of $G_{\sigma,+}^{\frac{p}{2}}$ in terms of $n$, $f$, $\Gamma_0$, $M_0$, $\varepsilon$, $\sigma$ and $p$. This suffices to apply the Stampacchia iteration argument.

\section{Exscribed curvature pinching}\label{sec_exsc}

We now turn attention to the proof of Theorem \ref{exscribed_improves}. We begin by setting $G_1 := \max\{-\underline k,0\}$ and $G_2:= 2\Theta F + H -|A|$, where $\Theta=\Theta(n,f,\Gamma_0,\Upsilon)$ is so large that $|A|-H\leq \Theta F$. We also ask that $\Theta \geq \Upsilon$ to ensure $G_1 \leq \Theta F \leq G_2$. Let $G = G_1^2/G_2$. 

A straightforward computation yields (cf. \cite[Proposition 11]{Brend15})
\begin{align}
\label{exsc_identity}
\partial_t G_1 \leq  \langle \nabla G_1, \mathcal T (\nabla F)\rangle + \frac{1}{2}F |\mathcal T(\nabla G_1)|^2 
\end{align}
almost everywhere in $\undersl U$. Setting 
\begin{align*}
\omega &:= \Delta_F G_1 - 2 \langle \nabla G_1 , \mathcal T (\nabla G_1)\rangle_F - \langle \nabla G_1, \mathcal T (\nabla F) \rangle - \frac{1}{2}F |\mathcal T(\nabla G_1)|^2,
\end{align*}
we combine \eqref{exsc_identity} with \eqref{exsc_evol} to obtain
\begin{align*}
(\partial_t - \Delta_F) G_1 &\leq |A|^2_F G_1  - \max\{ \omega + |A|^2_F G_1, 0 \} -2\langle \nabla G_1, \mathcal T(\nabla G_1)\rangle_F ,
\end{align*}
and thus compute
\begin{align}
\label{eq:exsc_Gevol}
(\partial_t - \Delta_F) G & \leq |A|^2_F G  - 2\frac{G}{G_1} \max\{ \omega + |A|^2_F G_1, 0 \} \notag\\
&\;\;\;\; - 4\frac{G}{G_1} \langle \nabla G_1, \mathcal T(\nabla G_1)\rangle_F- \frac{G}{G_2} Q_{G_2, F} (\nabla A),
\end{align}
which holds distributionally on $\undersl U$. Note that, writing $N(Z) = \tr(Z)-|Z|$,
\begin{align*}
Q_{G_2, F}(\nabla A) = (\dot N^{kl} \ddot f^{pq, rs} - \dot f^{kl} \ddot N^{pq,rs}) \nabla_k A_{pq} \nabla_l A_{rs} \geq 0
\end{align*}
since $f$ is convex and $N$ is concave and non-decreasing. 

We now set $G_\sigma := (G-\varepsilon F)F^{\sigma  -1} - K$ and $G_{\sigma, +} := \max\{G_\sigma, 0\}$, where $\varepsilon >0$ and $K = K(n, f, \Gamma_0, M_0, \varepsilon)$ is so large that
\[\kappa_1 \geq -\frac{\varepsilon}{2} F - K\min\{1,F\}.\] 
That such a choice of $K$ is possible was proven in \cite{ALM14}. Note that $G_{\sigma} \geq 0$ implies $G_1 > 0$ and 
\begin{align*}
G_1 + \kappa_1 &\geq G - \frac{\varepsilon}{2}F - K \min \{1, F\}\\
		&\geq  \frac{\varepsilon}{2}F + K F^{1-\sigma} - K\min\{1, F\}\\
		&\geq \frac{\varepsilon}{2}F\,.
\end{align*}
So $\supp(G_{\sigma, +}) \subset \undersl U$. In particular, wherever $G_\sigma$ is positive, \eqref{eq:exsc_Gevol} holds and we have 
{\small\begin{align*}
(\partial_t - \Delta_F)G_\sigma &\leq \sigma |A|^2_F (G_\sigma + K) -2F^{\sigma-1}\frac{G}{G_1}\max\{ \omega + |A|^2_F G_1, 0 \}\\
&\;\;\;\; -GF^{\sigma-1} \left( 4 \left \langle \mathcal T (\nabla G_1) , \frac{\nabla G_1}{G_1}\right \rangle_F + \frac{Q_{G_2, F} (\nabla A)}{G_2}\right) \\
&\;\;\;\; +2(1-\sigma) G_\sigma \left\langle \frac{\nabla G_\sigma}{G_\sigma}, \frac{\nabla F}{F}\right \rangle_F - \sigma(\sigma-1)G_\sigma \frac{|\nabla F|_F^2}{F^2}.
\end{align*}}\hspace{-1.5mm}
Estimating now
\begin{align}
\label{eq:exsc_goodgrad}
4\left \langle\mathcal T(\nabla G_1), \frac{\nabla G_1}{G_1}\right \rangle_F \geq \gamma_1(n,f,\Gamma_0, \Upsilon) \frac{|\nabla \undersl k |^2}{F^2},
\end{align}
we obtain
{\small\begin{align*}
(\partial_t - \Delta_F)G_\sigma &\leq \sigma |A|^2_F (G_\sigma + K) -2\frac{G_\sigma}{G_1}\max\{ \omega + |A|^2_F G_1, 0 \}\\
&\;\;\;\; -\gamma_1 G_\sigma \frac{|\nabla \undersl k|^2}{F^2}+2(1-\sigma) G_\sigma \left\langle \frac{\nabla G_\sigma}{G_\sigma}, \frac{\nabla F}{F}\right \rangle_F\\
&\;\;\;\;  - G_\sigma\left(\frac{Q_{G_2,F}(\nabla A)}{G_2}+\frac{\sigma}{2} \frac{|\nabla F|_F^2}{F^2}\right).
\end{align*}}\hspace{-1.5mm}
Exactly as in Lemma \ref{lem:insc_goodgrad}, the final term can be estimated by
\[-G_\sigma\left(\frac{Q_{G_2,F}(\nabla A)}{G_2} + \frac{\sigma}{2} \frac{|\nabla F|_F^2}{F^2} \right)\leq -5\sigma \gamma_2 G_\sigma \frac{|\nabla A|^2}{F^2},\]
where $\gamma_2 = \gamma_2(n, f, \Gamma_0)$. Applying Young's inequality to the inner product term, we arrive at 
\begin{align*}
(\partial_t - \Delta_F)G_\sigma &\leq \sigma |A|^2_F (G_\sigma + K) -2\frac{G_\sigma}{G_1}\max\{ \omega + |A|^2_F G_1, 0 \}\\
&\;\;\;\; -\gamma_1 G_\sigma \frac{|\nabla \undersl k|^2}{F^2} - 4\sigma\gamma_2 G_\sigma  \frac{|\nabla A|^2}{F^2} + \frac{C}{\sigma} \frac{|\nabla G_\sigma|^2}{G_\sigma},
\end{align*}
where $C = C(n,f,\Gamma_0)$. Finally, since $2\sigma < 1$,
\begin{align*}
- 2\frac{G_\sigma}{G_1}\max\{ \omega + |A|^2_F G_1, 0 \} &\leq - 4\sigma \frac{G_\sigma}{G_1}\max\{ \omega + |A|^2_F G_1, 0 \} \\
		&\leq -4\sigma \frac{G_\sigma}{G_1} \omega - 4\sigma |A|^2_FG_\sigma, 
\end{align*}
so we have 
\begin{align}
\label{eq:exsc_pinchevol}
(\partial_t - \Delta_F)G_\sigma &\leq \sigma |A|^2_F (K - 3G_\sigma) -4\sigma\frac{G_\sigma}{G_1}\omega  -\gamma_1 G_\sigma \frac{|\nabla \undersl k|^2}{F^2} \notag\\
&\;\;\;\; - 4\sigma\gamma_2 G_\sigma  \frac{|\nabla A|^2}{F^2} + \frac{C}{\sigma} \frac{|\nabla G_\sigma|^2}{G_\sigma}
\end{align}
in $\supp(G_{\sigma, +})$, in the sense of distributions. This will suffice to carry out the iteration argument. 

\begin{proposition}
\label{exscribed_Lp}
There exist constants $\ell=\ell(n,f,\Gamma_0,\Upsilon,\varepsilon)>0$ and $C = C(n, f, \Gamma_0)$ such that
\[\frac{d}{dt} \int (G_{\sigma,+}^p+\sigma K^p C) \leq - \sigma p\int |A|^2_FG_{\sigma,+}^p\]
for almost every $t \in (0,T)$ whenever $p\geq \ell^{-4}$ and $\ell^{-1}p^{-1}\leq \sigma \leq \ell p^{-\frac{1}{2}}$. 
\end{proposition}
\begin{proof}
The inequality \eqref{eq:exsc_pinchevol} implies
\begin{align*}
\frac{d}{dt}\int G_{\sigma,+}^p &\leq p \int G_{\sigma,+}^{p-1} \Delta_F  G_{\sigma}  - 4\sigma p \int \frac{G_{\sigma,+}^p}{G_1} \omega -\gamma_1 p \int G_{\sigma,+}^p \frac{|\nabla \undersl k|^2}{F^2}\\
&\;\;\;\; -4\gamma_2 \sigma p \int G_{\sigma, +}^{p}\frac{ |\nabla A|^2}{F^2} +C\sigma^{-1}p\int G_{\sigma,+}^{p-2} |\nabla G_{\sigma, +}|^2\\
 &\;\;\;\; +\sigma p\int G_{\sigma, +}^{p-1}|A|^2_F(K-3G_\sigma)- \int G_{\sigma, +}^p HF.
\end{align*}
We estimate the first term on the right exactly as in \eqref{eq:insc_hard2}, and the last term can be dropped. Wherever $G_\sigma \geq 0$ we have $\kappa_i-\underline k\geq \frac{\varepsilon}{2}F$, so
\begin{align*}
-\omega &\leq -\Delta_F G_1 + \frac{C}{F}\left(|\nabla \undersl k|^2 + |\nabla \undersl k| |\nabla F| \right)\\
&\leq - \Delta_F G_1+ C\frac{|\nabla \undersl k|^2}{F} + \frac{\gamma_2}{4} G_1\frac{|\nabla A|^2}{F^2},
\end{align*}
where $C=C(n,f,\Gamma_0,\Upsilon, \varepsilon)$. Integrating by parts yields
\begin{align*}
- \int \frac{G_{\sigma,+}^p}{G_1} \Delta_F G_1 &\leq - \int G_{\sigma,+}^p \frac{|\nabla \undersl k|^2_F}{G_1^2}\\
&\;\;\;\; + C \int G_{\sigma,+}^p \frac{|\nabla \undersl k |}{G_1}\left(p \frac{|\nabla G_\sigma|}{G_\sigma}+\frac{|\nabla A|}{F} \right)\\
&\leq C\int G_{\sigma, +}^p\left( (1+p^\frac{1}{2})\frac{|\nabla \undersl k|^2}{F^2} + p^\frac{3}{2} \frac{|\nabla G_\sigma|^2}{G_\sigma^2}\right)\\
&\;\;\;\; +\frac{\gamma_2}{4} \int G_{\sigma, +}^p\frac{|\nabla A|^2}{F^2}
\end{align*}
where $C$ is a different constant which depends on $n$, $f$, $\Gamma_0$, $\Upsilon$ and $\varepsilon$. Collecting these estimates, we see that there is a constant $B<\infty$ depending only on $n$, $f$, $\Gamma_0$, $\Upsilon$ and $\varepsilon$ such that
\begin{align}
\label{eq:exsc_fin}
\frac{d}{dt} \int G_{\sigma, +}^p &\leq -(p(p-1) - B(\sigma^{-1} p + \sigma p^\frac{5}{2})\int G_{\sigma, +}^{p-2} |\nabla G_{\sigma}|^2 \notag\\
&\;\;\;\; - (\gamma_1p - B\sigma p (1+ p^\frac{1}{2})) \int G_{\sigma, +}^p \frac{|\nabla \undersl k|^2}{F^2}- \gamma_2 \sigma p \int G_{\sigma, +}^p \frac{|\nabla A|^2}{F^2} \notag \\
&\;\;\;\; + \sigma p\int |A|^2_FG^{p-1}_{\sigma, +}(K-2G_{\sigma}) - \sigma p\int |A|^2_FG_{\sigma,+}^p\,.
\end{align}
Finally, we estimate
\begin{align*}
p G_{\sigma, +}^{p-1} (K - 2G_{\sigma}) &\leq (K^p + (p-1)G_{\sigma, +}^p - 2p G_{\sigma,+}^p ) \leq K^p\,. 
\end{align*}
Taking $\ell$ sufficiently small, and $p$ sufficiently large, the claim follows.
\end{proof}

The remainder of the proof is the same as that of Theorem \ref{inscribed_improves}.

\section{Ancient solutions}

We begin this section by stating some geometric conditions that force an ancient solution of \eqref{eq:CF} to satisfy the decay condition \eqref{eq:volume_decay}, i.e.
\[\int_t^0 \hspace{-2.5mm} \int F \leq C(1-t)^r\quad\text{for all}\quad t<0\]
for constants $C > 0$ and $r \geq \frac{n+1}{2}.$

\begin{lemma}
\label{lem:area_conds}
Let $f:\Gamma\to\R$ be an admissible speed, $X:M\times (-\infty,1)\to \mathbb{R}^{n+1}$ be a solution of \eqref{eq:CF} and suppose that one of the following conditions holds:
\begin{enumerate}[(i)]
\item $X_t$ is an embedding bounding the region $\Omega_t$ and there are constants $C<\infty$ and $r \geq \frac{n+1}{2}$ such that
\begin{equation*}
|\Omega_t| \leq C(1-t)^r
\end{equation*}
for all $t \leq 0$.
\item There are constants $C<\infty$ and $r \geq \frac{n}{2}$ such that $F\leq CH$ and
\begin{equation*}
\mu_t(M) \leq C(1-t)^r
\end{equation*}
for all $t \leq 0$.
\item $X_t$ is an embedding and there is a constant $C<\infty$ such that
\[
\max_{M\times\{t\}}|F|\leq C
\]
for all $t\leq 0$.
\item $\Gamma=\Gamma_+$ and $f$ is inverse-concave\footnote{Note that convex speeds $f$ defined on $\Gamma_+$ are automatically inverse-concave.}.
\item There are constants $p\geq n$ and $C<\infty$ such that $F\leq CH$ and
\[\int H^p \leq C\]
for all $t \leq 0$.
\item There are constants $k \geq 0$ and $C<\infty$ such that $F\leq CH$ and
\[\left\langle \frac{\nabla H}{H}, \frac{\nabla F}{F} \right\rangle \leq \frac{1}{n+k-1}\left (|\AA|^2 + \frac{k}{n(n+k)}H^2 \right)\]
for all $t\leq 0$. 
\item There is a constant $C<\infty$ such that $C^{-1}F\leq H\leq CF$ and
\[
F\leq C(1-t)^{-\frac{1}{2}}
\]
for all $t<0$.
\end{enumerate}
Then $X$ satisfies \eqref{eq:volume_decay}. 
\end{lemma}
\begin{remark}
If $f:\Gamma\to\R$ is concave then we can always bound $F\leq CH$ since $f(z)\leq C(n,f)\tr(z)$, where $C(n,f):=f(1,\dots,1)/n$ and $\tr(z):=z_1+\dots+z_n$. The opposite inequality holds if $f$ is convex. 
\end{remark}
\begin{proof}[Proof of Lemma \ref{lem:area_conds}]
(i) This follows directly from the formula
\[ \int_t^0 \hspace{-2mm} \int F = |\Omega_t| - |\Omega_0|.\]

(ii) Applying H\"{o}lder's inequality twice, we obtain
\begin{align*}
\int_t^0\hspace{-2mm} \int F(\cdot, s) \,d\mu_s \, ds &\leq \int_t^0  \mu_s(M)^\frac{1}{2} \left(\int F(\cdot, s)^2 \,d\mu_s\right)^\frac{1}{2} ds\\
		&\leq (-t)^\frac{1}{2} \left( \int_t^0 \mu_s(M) \int F(\cdot,s)^2 \,d\mu_s \, ds\right)^\frac{1}{2}.
\end{align*}
Bounding $F \leq CH$ for $t\leq 0$ and applying \eqref{area_decay} then gives
\begin{align*}
\left(\int_t^0\hspace{-2mm} \int F \, d\mu_s \, ds \right)^2 &\leq C(-t) \int_t^0 \mu_s(M) \int HF \,d\mu_s \, ds \\
		&= - C(-t) \int_t^0 \mu_s(M)\frac{d}{ds} \mu_s(M)\, d\mu_s \,ds\\
		&\leq -\frac{Ct}{2}\mu_t(M)^2
\end{align*}
from which \eqref{eq:volume_decay} follows by assumption.\\

(iii) As in \cite{HuSi15}, integrating the uniform bound for the speed implies a bound $\rho_+(t)\leq C(\sup_{M\times(-\infty,0]}|F|,\rho_+(0))(1-t)$ for the circumradius $\rho_+(t)$ of $M_t$ for $t<0$, and hence a bound for the enclosed volume $|\Omega_t|\leq C(n,\sup_{M\times(-\infty,0]}|F|,\rho_+(0))(1-t)^{n+1}$. We may then appeal to (i). \\

(iv) 
Taking $t_0 \to -\infty$ in the differential Harnack inequality \eqref{eq:Harnack}, we obtain $\partial_t F \geq 0.$ Thus $0<F(\cdot, t) \leq F(\cdot, 0)$ for all $t \leq 0$ and we may appeal to (iii). \\

(v) Bounding $F\leq CH$ for $t\leq 0$, \eqref{area_decay} and H\"{o}lder's inequality yield
\begin{align*}
-\frac{d}{dt} \mu_t(M) = \int HF\leq{}& C \int H^2\leq C\mu_t(M)^{1-\frac{2}{p}} \left(\int H^p \right)^\frac{2}{p}\,.
\end{align*}
By hypothesis, $\Vert H\Vert_{L^p}$ is uniformly bounded so that, rearranging, we obtain
\[\frac{d}{dt} \mu_t(M)^\frac{2}{p} \geq -C'\]
for $t<0$. Integrating yields
\[
\mu_t(M)\leq C''(1-t)^{\frac{p}{2}}
\]
for $t<0$ and \eqref{eq:volume_decay} then follows from part (ii). \\

(vi) The evolution equations \eqref{area_decay} and
\[\partial_t H = \Delta F + |A|^2F\]
(see \cite{HuPol}) yield
\begin{align*}
\frac{d}{dt} \int H^{n+k} 
& = -(n+k) \int H^{n+k-2} \Bigg [ (n+k-1)\langle \nabla H, \nabla F\rangle\\
 &\qquad\qquad\qquad\qquad\qquad- HF\left(|\AA|^2 + \frac{k}{n(n+k)}H^2\right)\Bigg ]\,.
\end{align*}
By assumption, the right hand side is non-negative. Integrating therefore yields
\[\int H^{n+k}\,d\mu_t\leq \int H^{n+k}\,d\mu_0\]
for $t\leq 0$ and we can now appeal to condition (v). \\

(vii) The evolution equation \eqref{area_decay} for the area and the hypotheses yield
\[
-\frac{d}{dt} \mu_t(M)=\int HF \leq C\int F^2 \leq C^3\frac{\mu_t(M)}{1-t}\,,
\]
which integrates to 
\[
\mu_t(M)\leq  \mu_0(M)(1-t)^{C^3}\,.
\]
The claim now follows from (ii).
\end{proof}

We can now prove the sphere characterisation, Theorem \ref{thm:sphere_char}.

\begin{proof}[Proof of Theorem \ref{thm:sphere_char}]
Fix $t \in (-\infty, 0]$ and denote the Gauss curvature of $M_t$ by $\mathfrak K$. Since $M_t$ is convex, $\nu$ is a diffeomorphism and 
\[\int_{M} \mathfrak K \, d\mu_t = c(n),\]
where $c(n)$ is the surface area of the unit sphere in $\mathbb{R}^{n+1}$. Uniform convexity then implies (cf. \cite{HuSi15})
\[\int_{M} F^n d\mu_t 
\leq C(n,f,\Gamma_0)\int_{M} \mathfrak K \, d\mu_t\leq C(n,f,\Gamma_0)\,.\]
It now follows from Lemma \ref{lem:area_conds} that $X$ satisfies \eqref{eq:volume_decay}.

Consider first the case that the speed is given by a concave function of the principal curvatures. Set
\[
\varphi := \int G_{\sigma,+}^p\,,
\]
where $G_\sigma$ is the modified curvature function introduced in the proof of Theorem \ref{cylindrical_ests} (with $m=0$ and any $\varepsilon >0$). Then, recalling \eqref{cyl_Lp_ancient}, we choose $\ell>0$ such that
\begin{equation}
\label{anc_lp}
\frac{d}{dt}\varphi\leq-\int |A|^2 G_{\sigma,+}^p
\end{equation}
for $p\geq \ell^{-1}$ and $\sigma = \ell p^{-\frac{1}{2}}$. Set $\gamma = \frac{2}{\sigma p +1}$ and choose $p$ so large that $\gamma <1$. H\"{o}lder's inequality and the estimate $G_\sigma\leq C(n,f,\Gamma_0)F^\sigma$ imply
\begin{align*}
\varphi &\leq \int F^{\sigma p \gamma}G_{\sigma, +}^{p(1-\gamma)} \\
		&\leq  \int F^{\sigma p \gamma -2(1-\gamma)}(F^2 G_{\sigma, +}^p)^{1-\gamma} = \int F^\gamma (F^2 G_{\sigma, +}^p)^{1-\gamma} \\
		&\leq \left( \int F\right)^\gamma \left( \int F^2 G_{\sigma, +}^p\right)^{1-\gamma}
\end{align*}
Combining this with \eqref{anc_lp} and estimating $F \leq C(n, f, \Gamma_0) |A|$, we obtain 
\begin{align*}
\frac{d}{dt} \varphi &\leq -C\int F^2G_{\sigma,+}^p\\
		&\leq -C \left(\int F \right)^{-\frac{\gamma}{1-\gamma}} \varphi^\frac{1}{1-\gamma} = -C\left(\int F \right)^{-\beta} \varphi^{1+\beta},
\end{align*}
with $\beta := \frac{2}{\sigma p -1}$, which we take to be positive. 

Suppose now that $\varphi(s) >0$ for some $s \in (-\infty,0]$, so that $\varphi(t) \geq \varphi(s) > 0$ for all $t < s$. The last inequality then implies 
\[\frac{d}{dt} \varphi^{-\beta} = - \beta \varphi^{-1-\beta} \frac{d}{dt} \varphi \geq C\beta \left(\int F \right)^{-\beta},\]
which we integrate in time to obtain 
\begin{align*}
\varphi^{-\beta} (s) \geq C \beta \int_t^s \left(\int F \right)^{-\beta} \geq C\beta (s-t)^{1+\beta} \left(\int_t^s \hspace{-2mm}\int F\right)^{-\beta}.
\end{align*}
Combined with the assumption \eqref{eq:volume_decay}, this gives
\[\varphi^{-\beta}(s) \geq C\beta \frac{(s-t)^{1+\beta}}{(1-t)^{\beta r}}.\]
Choosing $p$ so large that $1 + \beta(1-r) >0$ and taking $t \to - \infty$ then yields a contradiction. We conclude that $\varphi\equiv 0$ for $t<0$. Since $\varepsilon$ was arbitrary, it follows that $G\equiv 0$ for $t<0$; i.e. $M_t$ is umbilic. The claim follows. The argument for convex speeds and surface flows proceeds similarly, with $G_\sigma$ instead given by the pinching functions devised in \cite{AnLa14} (with $m=0$) and \cite[\S 5.3.1]{L} (cf. \cite{An10}) respectively.

\end{proof}

Theorem \ref{thm:sphere_char_typeI} now follows from a blow-up argument (cf. \cite[\S 4]{HuSi15}).

\begin{proof}[Proof of Theorem \ref{thm:sphere_char_typeI}]
If the solution is uniformly pinched, in the sense that $\kappa(M\times(-\infty,0])\subset\Gamma_0\Subset \Gamma_+$, then the claim follows from Theorem \ref{thm:sphere_char}. Otherwise, there is a sequence of points $(x_k,t_k)\in M\times(-\infty,0]$ with $t_k\to-\infty$ such that
\[
\frac{\kappa_1}{|A|}(x_k,t_k)\to 0\quad\text{as}\quad k\to\infty\,.
\]
Note that bounded speed ratios imply type-I curvature decay, since, by an elementary ODE comparison argument,
\[
C^{-1}\min_{M\times\{t\}}F\leq \frac{1}{\sqrt{1-t}}\leq C\max_{M\times\{t\}}F
\] 
for some $C=C(n,f,\Gamma_0,M_0)$. Thus, we can assume in both cases that
\[
\sqrt{-t}F(\cdot,t)\leq C<\infty
\]
for $t<0$ for some $C<\infty$. Integrating yields $\mathrm{diam}(M_t)\leq C'\sqrt{-t}$ for $t<0$, where $C'$ depends on $C$ and $\mathrm{diam}(M_0)$. By the technical condition $\kappa(M\times(-\infty,0])\subset\Gamma_0\Subset\Gamma$, we can also estimate
\[
|A|\leq C''F
\]
for all $t\in (-\infty,0]$ for some $C''=C''(n,f,\Gamma_0)$. Consider now the rescaled flow
\[
X_k(x,t):=\lambda_kX(x,\lambda_k^{-2}t)\,,\quad t\in[-2,-1]\,,
\]
where $\lambda^{-1}_k:=\sqrt{-t_k}$. The type-I condition and the diameter bound imply uniform curvature and diameter bounds for the sequence. Since $\kappa(M\times(-\infty,0])\subset\Gamma_0\Subset\Gamma$, a well-known argument implies that a subsequence converges locally uniformly in $C^\infty(M\times[-2,-1])$ to a smooth compact solution of \eqref{eq:CF} (see, for example, \cite[Proposition 4.2 and Appendix C]{L} or \cite[Sections 10 and 13]{AMY}). 
Since, $\kappa(M\times(-\infty,0])\subset\Gamma_0\Subset\Gamma$, the limit satisfies $F>0$. Since there is a point on the limit satisfying $\kappa_1=0$, the strong maximum principle implies that $\kappa_1\equiv 0$ and the limit splits off a line (see \cite[Theorems 4.21 and 4.23]{L} and \cite[Theorem A.1]{BoLa16}). This contradicts compactness.
%
\end{proof}

Theorem \ref{thm:sphere_char_Harnack} follows as in \cite{HuSi15} by way of Andrews' Harnack inequality.
\begin{proof}[Proof of Theorem \ref{thm:sphere_char_Harnack}]
First note that, by \cite[Lemma 4.4]{HuSi15}, a bounded isoperimetric ratio (\ref{hyp:rii}) implies bounded eccentricity (\ref{hyp:ecc}). Next, by comparing our solution with appropriate shrinking sphere solutions, we observe that
\[
\rho_-(M_t)\leq C\sqrt{1-t}
\]
for $t<0$, where $C=C(n,f,M_0)$. Thus, bounded eccentricity (\ref{hyp:ecc}) implies bounded rescaled diameter
\begin{align}\label{eq:diameterbound}
\frac{\mathrm{diam}(M_t)}{\sqrt{1-t}}\leq \sqrt{C'}\,.
\end{align}

Integrating Andrews' differential Harnack inequality \cite{An94b} we obtain, for a convex ancient solution of \eqref{eq:CF} with inverse-concave speed,
\begin{align*}
F(x,t)\leq F(y,s)\,\mathrm{exp}\left(\frac{\mathrm{diam}^2{M_t}}{4(s-t)}\right) \;\; \text{for all }\; x,\, y\in M \;\text{ and }\; t<s<1\,.
\end{align*}
Applying \eqref{eq:diameterbound} and setting $s=(1+t)/2$ yields
\[
\max_{M\times \{t\}}F\leq \mathrm{e}^{\frac{C'}{2}}\min_{M\times\{\frac{1+t}{2}\}}F\leq C''(1-t)^{-\frac{1}{2}}\,.
\]
The claim now follows from Theorem \ref{thm:sphere_char_typeI}.
\end{proof}

The convexity estimate, Theorem \ref{thm:anc_cnvx}, is proved by a similar argument to that of Theorem \ref{thm:sphere_char}.

\begin{proof}[Proof of Theorem \ref{thm:anc_cnvx}] To see that the solution is convex, we apply the arguments of Theorem \ref{thm:sphere_char} to the pinching function constructed in \cite[Section 3]{ALM14} in case $f$ is convex, and to the pinching function constructed in \cite[Section 3]{ALM15} in case $n=2$. 
\end{proof}

With the convexity estimate in place, we can also obtain a sharp estimate for the exscribed curvature.
\begin{theorem}
\label{anc_extNC}
Let $f:\Gamma\to\R_+$ be a convex admissible speed. Let $X:M \times (-\infty, 1) \to \mathbb{R}^{n+1}$ be a compact ancient solution of \eqref{eq:CF} satisfying  $\kappa(M\times(-\infty,0])\subset\Gamma_0\Subset \Gamma$, the decay condition \eqref{eq:volume_decay} and the exterior non-collapsing condition $\underline k\geq-\Upsilon F$ for some $\Upsilon<\infty$. Then $\underline k>0$. In particular, $M_t$ bounds a strictly convex body for all $t$.
\end{theorem}

\begin{proof}
By Theorem \ref{thm:anc_cnvx}, the solution is convex, so we may take $K = 0$ in the definition of the pinching function $G_\sigma$ constructed to prove the exscribed curvature estimate in Section \ref{sec_exsc}. For $p$ sufficiently large and $\sigma  \sim \ell  p^{-\frac{1}{2}}$, \eqref{eq:exsc_fin} then implies an estimate of the form \eqref{anc_lp}. Continuing as in the proof of Theorem \ref{thm:sphere_char}, we conclude that $G_\sigma$ vanishes identically, and thereby $\underline k \geq 0$. The strict inequality follows from the strong maximum principle since the circumradius of $M_t$ is finite for all $t$.
\end{proof}

We now turn our attention to the cylindrical estimates. 

\begin{proposition}\label{prop:m_convex_anc}
Let $f:\Gamma\to\R_+$ be an admissible speed and let $X:M\times(-\infty, 1) \to  \mathbb{R}^{n+1}$ be a compact ancient solution of \eqref{eq:CF} satisfying $\kappa(M\times(-\infty,0])\subset \Gamma_0 \Subset \Gamma_{m+1}\cap \Gamma$ for some $m\in\{0,\dots,n-1\}$ and the volume decay estimate \eqref{eq:volume_decay}. 
\begin{itemize} 
\item[--] If $f$ is concave, then 
\begin{equation}
\label{cyl_anc_concave}
\kappa_n-c_m F \leq 0\quad\text{on}\quad M \times (-\infty ,1)\,.
\end{equation}
\item[--] If $f$ is convex, then 
\begin{equation}
\label{cyl_anc_convex}
\kappa_1+\dots+\kappa_{m+1}-c_m F \geq 0\quad\text{on}\quad M \times (-\infty ,1)\,.
\end{equation} 
\item[--] If $f$ is concave and, in addition, $X$ is interior non-collapsing (i.e. 
if $\overline k \leq \Lambda F$ for some $\Lambda<\infty$) then
\begin{equation}
\label{insc_anc}
\overline k - c_m F \leq 0\quad \text{on}\quad M \times (-\infty ,1)\,.
\end{equation}
\end{itemize}
\end{proposition}

\begin{proof}
To obtain \eqref{cyl_anc_concave} for concave speeds, we proceed as in the proof of Theorem \ref{thm:sphere_char} with $G_\sigma$ given by the pinching function used in Theorem \ref{cylindrical_ests} (this time also for non-zero $m$). 

To obtain \eqref{cyl_anc_convex} for convex speeds, we instead use the pinching function introduced at the beginning of Section 4 in \cite{AnLa14} for $m\leq n-2$. The case $m=n-1$ follows from the convexity estimate.

The proof of the optimal inscribed curvature pinching \eqref{insc_anc} also follows the proof of Theorem \ref{thm:sphere_char}, using with the pinching function devised in \eqref{insc_pinching}, since, with \eqref{cyl_anc_concave} in place, we can take $K=0$ in \eqref{insc_pinching}, and thereby obtain the estimate \eqref{anc_lp} from \eqref{insc_Lp_ancient1} (for $m\leq n-2$) or \eqref{insc_Lp_ancient2} (for $m=n-1$). 
\end{proof}


In fact, the strong maximum principle yields strict inequality, at least for strongly concave speeds.

\begin{proposition}\label{prop:strict_m_convex_anc}
If $M$ is connected and $M_t$ is embedded for all $t<0$ and if $F$ is strictly concave in non-radial directions, then each of the inequalities \eqref{cyl_anc_concave} and \eqref{insc_anc} is strict, unless $m=0$ (in which case, $M_t$ is a shrinking sphere).
\end{proposition}
\begin{proof}
To obtain the strict inequalities, we appeal to the strong maximum principle. First consider  \eqref{cyl_anc_concave}. Recall that the largest principal curvature $\kappa_n$ satisfies
\begin{align*}
(\partial_t-\Delta_F)\kappa_n\leq{}& |A|^2_FF+\ddot F^{pq,rs}\nabla_nA_{pq}\nabla_nA_{rs}\\
{}&-2\sum_{k=1}^n\sum_{\kappa_i<\kappa_n}\frac{\dot F^k}{\kappa_n-\kappa_i}(\nabla_kA_{in})^2
\end{align*}
in the viscosity sense. This is proved by a well-known argument using the `two-point function' $K:TM\to\R$, $(x,y)\mapsto A_x(y,y)/g_x(y,y)$ as a smooth lower support for $\kappa_n$ (Cf. \cite[Theorem 3.2]{An07}). It follows that
\begin{align}
(\partial_t-\Delta_F)\frac{\kappa_n}{F}\leq{}&\frac{2}{F}\left\langle\nabla\frac{\kappa_n}{F},\nabla F\right\rangle+\frac{1}{F}\ddot F^{pq,rs}\nabla_nA_{pq}\nabla_nA_{rs}\nonumber\\
{}&-\frac{2}{F}\sum_{k=1}^n\sum_{\kappa_i<\kappa_n}\frac{\dot F^k}{\kappa_n-\kappa_i}(\nabla_kA_{in})^2\label{eq:SMPconcave}
\end{align}
in the viscosity sense. Since $F$ is concave, the strong maximum principle \cite{DL04} implies that equality $\kappa_n=c_mF$ can only be attained at some point $(x_0,t_0)\in M\times(-\infty,1)$ if it holds identically in $M\times(-\infty,t_0]$. In that case, $\kappa_n$ is smooth (for $t<t_0$ as we henceforth implicitly assume) and \eqref{eq:SMPconcave} implies that
\[
\nabla A_{nn}=\nabla\kappa_n\equiv c_m\nabla F\,,
\]
\[
\nabla A_{in}\equiv 0\quad \text{for all} \quad \kappa_i<\kappa_n
\]
and, since $F$ is strictly concave in non-radial directions,
\[
\nabla_nA=\mu A\quad\text{for some}\quad \mu:M\times(-\infty,t_0]\to\R\,.
\]
It follows that $\nabla_kF\equiv 0$ for all $k\neq n$, since
\[
c_m\nabla_kF=\nabla_kA_{nn}=\nabla_nA_{kn}=\mu A_{kn}=0\,.
\]
If $\mu(x,t)=0$, then 
\[
\nabla_nF=\dot F^{ij}\nabla_nA_{ij}=\mu\dot F^{ij}A_{ij}=\mu F=0
\]
at $(x,t)$ and we deduce that $\nabla F=0$ at $(x,t)$. If instead $\mu(x,t)\neq 0$ then for each $\kappa_i<\kappa_n$
\[
\mu\kappa_i=\mu A_{ii}=\nabla_nA_{ii}=\nabla_iA_{ni}=0\,.
\]
Since $\kappa_n=c_mF$, this means that $\kappa_1=\dots=\kappa_m=0$ and $\kappa_{m+1}=\dots=\kappa_n$. In particular, the multiplicity of $\kappa_n$ is equal to $n-m\geq 2$. Thus, applying the same argument to $\kappa_{n-1}$ yields $\nabla_kF\equiv 0$ for all $k\neq n-1$. We conclude again that $\nabla F=0$. It follows that each connected component of $M_t$ is a round sphere for $t\in(-\infty,t_0]$ and hence, by uniqueness of compact solutions of \eqref{eq:CF}, for $t\in (-\infty,1)$ \cite{Korevaar}. The claim follows.


Next, consider \eqref{insc_anc}. Suppose, contrary to the claim, that equality $\overline k=c_m F$ is attained at some point $(x_0,t_0)\in M\times(-\infty,1)$. Since
\[
(\partial_t-\Delta_F)\frac{\overline k}{F}\leq \frac{2}{F}\left\langle\nabla\frac{\overline k}{F},\nabla F\right\rangle
\]
in the viscosity sense, the strong maximum principle \cite{DL04} implies that $\overline k\equiv c_mF$ on $M\times (-\infty,t_0]$. In particular, $\overline k$ is smooth (for $t<t_0$ as we henceforth implicitly assume). Moreover, by the previous claim, $\kappa_n<\overline k$. The evolution equation \eqref{insc_evol} now implies
\[
0\equiv (\partial_t-\Delta_F)\frac{\overline k}{F}\leq -\frac{2}{F}\sum_{i=1}^n\frac{(\nabla_i\overline k)}{\overline k-\kappa_i}\leq 0\,.
\]
It follows that $\nabla\overline k\equiv 0$ and hence $\nabla F\equiv 0$ on $M\times(-\infty,t_0]$, which implies the claim as above.
\end{proof}


This leads to a further characterisations of the sphere for concave speeds.

\begin{proof}[Proof of Theorem \ref{thm:2conv_sphere}]
By Lemma \ref{lem:area_conds} (vii), the volume decay condition \eqref{eq:volume_decay} is satisfied. Thus, by Proposition \ref{prop:m_convex_anc}, 
\[\kappa_n - c_1 F \leq 0\]
on $M\times (-\infty, 0)$. This implies that $X$ is weakly convex, i.e. $\kappa_1 \geq 0$. Since $M$ is compact, the splitting theorem \cite[Theorem 5.1]{BoLa16} can be applied to show that the solution is in fact convex, as long as the restriction of $f$ to the set $\{z \in \mathbb{R}^n : z_1 = 0,\, z_2, \dots, z_n >0\}$ is inverse-concave. That this follows from inverse-concavity is proved in \cite[Remark 1 (6)]{AMY}. 
The claim now follows from Theorem \eqref{thm:sphere_char_typeI}.
\end{proof}

\begin{proof}[Proof of Theorem \ref{thm:strict_conv_sphere}]
The proof is similar to the proof of Theorem \ref{thm:sphere_char_typeI}
: By the type-I condition and the pinching assumption, the volume decay estimate \eqref{eq:volume_decay} needed to apply Proposition \ref{prop:m_convex_anc} holds. Combined with Proposition \ref{prop:strict_m_convex_anc}, this yields $\kappa_n-c_mF<0$, which, in particular, implies $\kappa_1+\dots+\kappa_m>0$. We claim that, unless $m=0$, $\kappa_n-c_mF\leq -\alpha F$, which implies that $\kappa_1+\dots+\kappa_m\geq \alpha'F$ (cf. \cite[Claim 4.2]{BoLa16}). This follows from a blow-down argument like the one applied in Theorem \ref{thm:sphere_char_typeI}
. Indeed, if the claim does not hold, then there is a sequence of points $(x_k,t_k)\in M\times(-\infty,0]$ with $t_k\to-\infty$ such that
\[
0>\frac{\kappa_n-c_mF}{F}(x_k,t_k)\to 0\quad\text{as}\quad k\to\infty\,.
\]
As in the proof of Theorem \ref{thm:sphere_char_typeI}
, after rescaling by a factor $\lambda_k:=1/\sqrt{-t_k}$ and passing to a subsequence we obtain a sequence of flows which converge to a compact limit flow $X_\infty:M\times(-2,-1]\to\R^{n+1}$ satisfying $\kappa_n-c_mF\leq 0$ with equality at some point $x_\infty\in M$ at time $t=-1$. Unless $m=0$, this contradicts the strong maximum principle for $\kappa_n$ as applied in Proposition \ref{prop:strict_m_convex_anc}. Thus, there is some $\alpha'<0$ such that $\kappa_1+\dots+\kappa_m\geq \alpha'F$ and we conclude from Proposition \ref{prop:m_convex_anc} that $\kappa_n-c_{m-1}F\leq 0$. Repeating the argument finitely many times, we conclude that $\kappa_n\leq c_0F$, which implies that the solution is a shrinking sphere.
\end{proof}

To prove the final sphere characterisation, we will require a lemma. Recall that a smooth embedding $X : \mathbb{R}^n \to \mathbb{R}^{n+1}$ is called a translating soliton of \eqref{eq:CF} if there is a unit vector $T \in \mathbb{R}^{n+1}$ such that
\[F(x) = - \langle \nu(x), T\rangle\quad\text{for all} \quad x\in \R^n\] 
This ensures that, up to composition with a time-dependent tangential reparameterisation, the family of immersions $X(\cdot, t) := X(\cdot) + tT$ solves \eqref{eq:CF} for all $t \in (-\infty,\infty)$.

\begin{lemma}
\label{lem:trans_gap}
Let $f: \Gamma \to \mathbb{R}$ be a concave admissable speed such that $\Gamma_+ \subset \Gamma $. Then, there is a positive constant $\delta_0 = \delta_0(n, f)$ such that any strictly convex translating soliton $X: \mathbb{R}^n \to \mathbb{R}^{n+1}$ of \eqref{eq:CF} which is uniformly two-convex (in the sense that $\kappa(\R^n)\subset \Gamma_0\Subset\Gamma_2\cap\Gamma$) satisfies
\[\sup_{\mathbb{R}^n} \frac{|\nabla A|^2}{F^4} \geq \delta_0\,.\]
\end{lemma}

\begin{proof}
Suppose to the contrary that we have a sequence of strictly convex, uniformly two-convex translators $X_j:\R^n\to\R^{n+1}$ satisfying
\begin{equation}
\label{eq:badgrad}
\sup_{\mathbb{R}^n} \frac{|\nabla A_j|^2}{F_j^4} \leq \frac{1}{j}\;.
\end{equation}
Arguing as in Lemma 2.1 of \cite{Ha15} we find that each $F_j$ attains its maximum $F_j =1$ at a unique point in $\mathbb{R}^n$. Translating and rotating, we can ensure that this maximum is attained at the origin in $\mathbb{R}^{n+1}$ for every $1 \leq j \leq \infty$, and that each $X_j$ has the same translation vector, $T_j=T \in \mathbb{R}^{n+1}$ say. As in \cite{BoLa16}, the gradient estimate \eqref{eq:badgrad} forces a subsequence of the $X_j(B_1(0))$ to converge in $C^2$ to a limit embedding $X_\infty:B_1(0) \to \mathbb{R}^{n+1}$, where $B_1(0)$ denotes the unit ball in $\mathbb{R}^n$. Since $F_j \leq 1$, \eqref{eq:badgrad} also implies that $|\nabla A_\infty| \equiv 0$, and $F_\infty(0) = 1$, so in fact $F_\infty \equiv 1$. Passing the translator condition to the limit then yields $1 \equiv F_\infty = -\langle \nu_\infty, T\rangle$, so the image of $X_\infty$ is contained in a hyperplane, which is a contradiction. 
\end{proof}

\begin{proof}[Proof of Theorem \ref{thm:grad_sphere}]
We first note that for $\delta_0 = \delta_0(n ,f, \Gamma_0)$ small enough, the gradient estimate \eqref{eq:strong_grad} forces $X$ to satisfy condition (4) of Lemma \ref{lem:area_conds} with $k=1$. Hence, the area of $M_t$ obeys \eqref{eq:volume_decay} and we may invoke Proposition \ref{prop:m_convex_anc} to conclude that $\kappa_n - c_1 F \leq 0$. As in the proof of Theorem \ref{thm:2conv_sphere}, this implies that $X$ is strictly convex.

Suppose $X$ is not a shrinking sphere so that, by Theorem \ref{thm:sphere_char_typeI}, we may assume that it fails to have type-I curvature decay. Arguing similarly as in \cite{HuSi15} (cf. Appendix B in \cite{L}) we find that there is a family of rescalings of $X$ converging to a weakly convex, uniformly two-convex translating soliton. Applying again the splitting theorem and cylindrical estimate, we see that this translator is in fact strictly convex. Taking $\delta_0$ to be at least as small as the constant in Lemma \ref{lem:trans_gap} then yields a contradiction. 

\end{proof}

\bibliographystyle{acm}
\bibliography{FNL_pinching_refs}

\end{document}